\documentclass[11pt]{siamart251216}

\usepackage{lipsum}
\usepackage{amsfonts}
\usepackage{graphicx}
\usepackage{epstopdf}
\usepackage{amsmath}
\newtheorem{remark}[theorem]{Remark}
\usepackage{algorithmic}
\usepackage{algorithm}
\ifpdf
  \DeclareGraphicsExtensions{.eps,.pdf,.png,.jpg}
\else
  \DeclareGraphicsExtensions{.eps}
\fi

\newsiamremark{hypothesis}{Hypothesis}
\crefname{hypothesis}{Hypothesis}{Hypotheses}
\newsiamthm{claim}{Claim}
\newsiamremark{fact}{Fact}
\crefname{fact}{Fact}{Facts}

\usepackage{amssymb, amsmath}
\usepackage{amscd}
\usepackage{verbatim}
\usepackage{enumerate}
\usepackage{color}
\usepackage{tikz}
\usepackage{colordvi}
\usepackage{graphics}
\usepackage{graphicx}

\allowdisplaybreaks

\newtheorem{alg}[algorithm]{Algorithm}

\numberwithin{equation}{section}

\newcommand{\beq}{\begin{equation}}
\newcommand{\eeq}{\end{equation}}
\newcommand{\beqa}{\begin{eqnarray}}
\newcommand{\eeqa}{\end{eqnarray}}

\usepackage[textsize=tiny]{todonotes}

\begin{document}

\ifpdf
\hypersetup{
  pdftitle={NGMRES FOR NSE},
  pdfauthor={Yunhui He and  Leo G Rebholz}
}
\fi

% Sets running headers as well as PDF title and authors
\headers{NGMRES FOR NSE}{Yunhui He and  Leo G Rebholz}

\title{Convergence analysis and proof of acceleration for NGMRES applied to the Picard iteration for Navier-Stokes equations}
\thanks{Submitted to the editors DATE.}

% Authors: full names plus addresses.
\author{Yunhui He\thanks{Department of Mathematics, University of Houston, 3551 Cullen Blvd, Room 641, Houston, Texas 77204-3008, USA 
  (\email{yhe43@central.uh.edu}).}
\and Leo G Rebholz\thanks{School of Mathematical and Statistical Sciences, Clemson University, Clemson, SC 29634, USA (\email{rebholz@clemson.edu}).  The work of LR was partially funded by Department of Energy grant DE-SC0025292}}

\maketitle

% REQUIRED

\begin{abstract}
We consider nonlinear GMRES (NGMRES) as an acceleration technique for the Navier–Stokes Picard iteration, a direction that has not previously been explored. We identify the optimal norm for the least squares optimization problem arising in the NGMRES algorithm, and establish a convergence analysis for NGMRES with general depth that proves NGMRES scales the Picard Lipschitz/contraction constant by the gain of the optimization problem.  To our knowledge, this is the first convergence proof for NGMRES that identifies the mechanism responsible for convergence acceleration.  Numerical experiments demonstrate that the convergence estimates are remarkably sharp.  In addition, NGMRES greatly improves the performance of the Picard iteration, even in cases where the unaccelerated iteration diverges.  
\end{abstract}
% REQUIRED

% REQUIRED  
% 65N22  Numerical solution of discretized equations for boundary value problems involving PDEs
% 35Q30  Navier-Stokes equations
% 65N30 Finite element, Rayleigh-Ritz and Galerkin methods for boundary value problems involving PDEs
% 65H10  Numerical computation of solutions to systems of equations
%========================================================================

 \section{Introduction}\label{s:intro}
Nonlinear partial differential equations (PDEs) play a central role in many real world applications, yet they are often difficult to solve analytically. This motivates the development of efficient iterative nonlinear solvers and the corresponding convergence theory. A variety of nonlinear acceleration methods have been proposed in the literature, including multisecant methods \cite{fang2009two}, nonlinear conjugate gradient methods \cite{hager2006survey}, Anderson acceleration (AA) \cite{Anderson65}, and nonlinear GMRES (NGMRES). This work focuses on NGMRES, which for a given fixed point operator $q(u)$ used to solve $g(u)=0$, is usually written in the following form (a formal definition is given in Section 2) for depth $m$ and iteration $k+1$:  
\[
 				u_{k+1} = q(u_k) + \sum_{i=0}^{m_k=\min\{k,m\}}\beta_i^{k+1} \left(q(u_k)- u_{k-i} \right),
\]
where $\beta^{k+1}=\big(\beta_0^{k+1},\beta_1^{k+1},\cdots, \beta_{m_k}^{k+1}\big)$ is obtained by solving the  least-squares optimization problem
 \begin{equation}\label{eq:min-NG2} 
 				\min_{\beta^{k+1}} \| g(q(u_k))+\sum_{i=0}^{m_k=\min\{k,m\}} \beta_i^{k+1} \left(g(q(u_k))-g(u_{k-i}) \right)\|^2.
 \end{equation}

The NGMRES idea was first introduced by Washio and Oosterlee in 1997, who proposed a Krylov subspace based acceleration technique analogous to GMRES for nonlinear multigrid on the finest level \cite{washio1997krylov} for difficult nonlinear elliptic scalar problems and for systems of nonlinear equations. In 2000, Oosterlee and Washio extended this approach to accelerate nonlinear multigrid on coarse levels \cite{oosterlee2000krylov} for recirculating incompressible flow problems. Since then, NGMRES type accelerations have been applied in several areas, including tensor decomposition \cite{sterck2012nonlinear, sterck2013steepest, sterck2021asymptotic}, CUTEst problems \cite{riseth2019objective}, and image restoration \cite{chang2003acceleration,savage2005improved}.  The comparison in \cite{sterck2021asymptotic} shows that AA and NGMRES are often competitive, and it is an open question as to what kinds of problems exist where one is 
significantly better than the other. Several new variants of NGMRES have been proposed recently. For example, \cite{he2026convergenceANG} introduces a periodic scheme in which the NGMRES step is applied only every $p$ fixed-point iterations. A generalized alternating NGMRES method is developed in \cite{he2025generalized} for PDE constrained optimization governed by transport equations, where comparisons with Anderson type methods indicate that alternating NGMRES performs favorably. These studies highlight the effectiveness of NGMRES in accelerating nonlinear fixed point iterations. 

To date, only limited progress has been made on the convergence theory for NGMRES. In \cite{greif2026convergence}, the authors show that, under certain conditions, full NGMRES (i.e., incorporating all previous iterates) applied to the Richardson iteration is equivalent to classical GMRES. In \cite{he2026ngmresprecon}, an analogous equivalence is established for preconditioned Richardson and preconditioned GMRES. For nonlinear problems, the work \cite{he2025convergence} initiated a convergence analysis of NGMRES applied to a simple contractive fixed point iteration. To the best of our knowledge, until this paper, there has been no convergence analysis for NGMRES which characterizes its mechanism for convergence acceleration.

Herein, we consider NGMRES applied to the Picard iteration 
for the Navier-Stokes equations (NSE), which are given by: on a domain $\Omega \subset \mathbb{R}^d, d=2,3$, 	
 	\begin{equation}\label{NS1}
 		\left\{\begin{aligned}
 			-\nu \Delta u+u\cdot\nabla u+ \nabla p&={f} \quad \text{in}~\Omega,\\
 			\nabla\cdot {u}&=0\quad \text{in}~\Omega,
 		\end{aligned}\right.
 	\end{equation}{black}
together with appropriate boundary conditions.  Here $u$ and $p$ represent the unknown velocity and pressure, respectively, $\nu>0$ is the kinematic viscosity which is inversely proportional to the Reynolds number $Re \sim \frac{1}{\nu}$, and ${f}$ is a given forcing.  The Picard iteration for the NSE is given by 
 	\begin{equation}\label{Pic1}
		\left\{\begin{aligned}
 		-\nu \Delta u_{k+1}+u_k\cdot\nabla u_{k+1}+ \nabla p_{k+1}&={f}, \\
 		\nabla\cdot {u}_{k+1}&=0.
 		\end{aligned}\right.
		\end{equation}	
		This iteration is globally stable and globally linearly convergent in the $H^1_0$ norm with rate $\kappa:=M\nu^{-2} \| f\|_{H^{-1}}$ (where $M$ is a constant dependent on the domain size defined in Section 2) provided $\kappa<1$ \cite{GR86,J16};  {\color{black} we also show herein that the sequence of nonlinear residuals for the Picard iteration also converge with rate $\kappa$ when $\kappa<1$.}  However, as $\kappa$ increases (i.e. $Re$ increases), the convergence of Picard slows and eventually fails \cite{PRTX25,PRX19} for $\kappa$ large enough.  Unfortunately, such failures occur for $\kappa$ well within the range of physically relevant problems \cite{Laytonbook}, and thus there is significant interest in improving its convergence properties.  To improve convergence of Picard iteration, many methods have been studied over the years \cite{ES96,FMW19,PRX19,farrell2021reynolds,thompson1989adaptive,ghia1982high,vanka1986block,washio1997krylov,oosterlee2000krylov}, with varying degrees of success.

{\color{black}The motivation for our choosing to study NGMRES-Picard for the NSE is to develop improved nonlinear solvers, since this system is used to model many application problems and thus improved solvers are always needed.  Our initial testing shows good results, including in comparisons where NGMRES outperforms AA when smaller depths are used.  Additionally, we aim to develop a convergence and acceleration theory for NGMRES by following the recent AA theory development, which began by first considering AA-Picard for the NSE in 2019 in \cite{PRX19}.  In that paper, the gain of the AA optimization problem was identified as the mechanism for AA convergence acceleration. The following year, this problem-specific theory was extended by the same authors to general contractive fixed point iterations \cite{EPRX20} and then a year later with sharper bounds to general fixed point iterations which need not be contractive \cite{PR21}.  Indeed, the authors of this paper have already begun study extending the results of this paper to more general settings.
We note Picard is chosen herein as the iteration for NSE over Newton since Picard is more robust with respect to both Reynolds number and initial guess \cite{PRTX25}.  However, one could combine Newton with the methods used herein by using NGMRES-Picard until the residual is sufficiently small and then switching to Newton. }

We now proceed formally for the moment in order to give a high level description of our main results, suppressing boundary conditions and moving the problem setting to the divergence-free velocity space (this will all be made mathematically precise in Sections 2 and 3). Take $V\subset H^1(\Omega)$ to be the divergence-free velocity space and define $q:V\rightarrow V$ as the solution operator of the (weak) Picard iteration and define the nonlinear residual operator $g:V\rightarrow V'$ by
\[
g(v) = -\nu \Delta v+v\cdot\nabla v - f,
\]
where $V'$ is the dual space of $V$, equality is meant in a weak sense, and the pressure and divergence-free constraint drop since $g:V\rightarrow V'$.  Then for $m\ge k$, NGMRES in this setting can be written (in the equivalent constrained optimization formulation) as
\[
 				u_{k+1} = \alpha^{k+1}_{k+1} q(u_k) + \sum_{i=k-m}^k \alpha^{k+1}_i u_i,
\]
where $\big(\alpha^{k+1}_{k+1},\ \alpha^{k+1}_{k},\cdots,\ \alpha^{k+1}_{k-m}\big)$ are obtained by solving the least-squares optimization problem
 \begin{equation}\label{eq:min-NG3} 
 				\min_{\sum_{i=k-m}^{k+1} \alpha_i^{k+1} = 1} 
				\| \alpha^{k+1}_{k+1} g(q(u_k))+ \sum_{i=k-m}^k \alpha^{k+1}_i g(u_i) \|_{V'}^2.
 \end{equation}
Note that $V'$ is used because $g$ maps into $V'$.  The use of this norm is critical for our convergence analysis, and moreover our computations show better performance when using $V'$ as the norm compared to  $\ell^2$.  
%We also show herein that the sequence of nonlinear residuals converges with rate $\kappa$
%$\{ g(q(u_k)) \}_k$ using Picard iterates converges with rate $\kappa$ if $\kappa<1$;
%\todo{Should we rephrase this "using Picard iterates converges with rate" since I do not think this is convergence, because the inequlity is related $gq(u)$ and $g(u)$ not $q(u)$ for latter.} 
%$ {\color{black}\| g( q( u_{k})) \|_{V'} } 
%\| g(\tilde u_{k+1}) \|_{V'}  
%\le \kappa \| g( u_{k}) \|_{V'}$ (i.e. the same rate %Picard converges in $H^1_0$ norm). 

Our main analytical results can now be stated as follows.  Define $\theta_{k+1}$ and $\gamma_{k+1}$ by
\begin{align*}
 \theta_{k+1} & := \frac{ \| \alpha^{k+1}_{k+1} g( q(u_k)) + \alpha^{k+1}_k g(u_k) + ... + \alpha^{k+1}_{k-m}g(u_{k-m}) \|_{V'}  }{ \| g(u_k) \|_{V'} }, \\
 \gamma_{k+1}& := \frac{ \| \alpha^{k+1}_{k+1} g(q(u_k) ) + \alpha^{k+1}_k g(u_k) + ... + \alpha^{k+1}_{k-m}g(u_{k-m}) \|_{V'}  }{ \| g(q(u_k)) \|_{V'} }.
 \end{align*}
Note that both $\theta_{k+1}\le 1$ and $\gamma_{k+1}\le 1$, and are only one in the special cases that the optimization problem returns $(0,1,0,...,0)$ or $(1,0,...,0)$, respectively. The only difference between $\theta_{k+1}$ and $\gamma_{k+1}$ is the denominator in their definitions. We prove the following convergence bounds:

\begin{align}
\| g(u_{k+1}) \|_{V'} & \le \gamma_{k+1} \kappa \| g(u_{k})   \|_{V'} + \mathcal{O}\left( \| g(u_{k})   \|_{V'}^2 + \| g(u_{k-1})   \|_{V'}^2 + ... + \| g(u_{k-m})   \|_{V'}^2\right), \label{result2} \\
\| g(u_{k+1}) \|_{V'} &\le \theta_{k+1} \| g(u_{k})   \|_{V'} + \mathcal{O}\left( \| g(u_{k})   \|_{V'}^2 + \| g(u_{k-1})   \|_{V'}^2 + ... + \| g(u_{k-m})   \|_{V'}^2\right). \label{result1} 
 \end{align}
 
The result \eqref{result2} shows that $\gamma_{k+1}$ scales the linear convergence rate of Picard from $\kappa$ to $\gamma_{k+1}\kappa$.  The definition $\gamma_{k+1}$ shows the `gain of the optimization problem' from the ratio of what the optimization gives to what Picard with no acceleration would give (or equivalently, the optimization problem return $(1,0,...,0)$.  Hence up to higher order terms, \eqref{result2} shows that NGMRES accelerates convergence precisely through the gain of the optimization problem.

The result \eqref{result1} is also quite interesting, and there is no AA-Picard analogue to it.  It shows first that, up to higher order terms, NGMRES reduces the nonlinear residual in the $V'$ norm.  The mechanism underlying NGMRES is that the least-square problem in NGMRES is derived by expanding $g(u_{k+1})$ around $u_k$ and truncating after the first-order term, where a finite-difference approximation is used for the first derivative; see \cite{washio1997krylov}.  Further, since $\kappa$ is typically a bound and not known precisely, $\theta_{k+1}$ gives a more precise estimate of the Lipschitz/contraction number at each iteration.  Indeed, our numerical tests show that once the nonlinear residual gets smaller, $\theta_{k+1}$ is a very sharp estimate of the convergence rate at each iteration.

To our knowledge, the estimates \eqref{result2}-\eqref{result1} are the first convergence results for NGMRES that show {\it how} NGMRES accelerates an iteration.  The result \eqref{result2} is similar at a high level to the proof of acceleration for AA in that it is a single step theory where the gain of the optimization problem drives the acceleration \cite{PR25,PRX19}.  Another similarity to the AA theory of \cite{PRX19} is that the constrained form of the optimization problem is natural for the analysis.  However, the analysis and results herein are fundamentally different from the AA theory for several reasons:   the residual expansion takes a different form since the extrapolation of NGMRES is different than AA, the nonlinear residual is used in NGMRES instead of the Picard fixed point residual, and the optimization norm here is $V'$ while for AA is $V$.  We note that in neither case is $\ell^2$ the best choice, and we show both numerically and analytically {\color{black} that using $\ell^2$ can be a poor choice for NGMRES-Picard, especially in 3D.  Using the $V'$ norm requires an additional Stokes solve at each nonlinear iteration and thus incurs a non-negligible cost compared to using $\ell^2$, however the `extra' cost is modest  (about 16\% in our tests) since it is the same Stokes matrix for each solve and can be pre-constructed and pre-factored.  Since using $V'$ is shown herein to give convergence in many fewer iterations than $\ell^2$ in certain settings, this extra cost is often easily justified.
}

The remainder of this paper is organized as follows. In Section 2 we give notation and mathematical preliminaries, and in Section 3 we {\color{black} prove} the NGMRES convergence result.  In section 4, we give results for several numerical tests that illustrate the theory, show the convergence bound is quite sharp, compare using $V'$ and $\ell^2$ as the least squares optimization norms, and mesh independence of convergence for our acceleration method.  Conclusions and future directions are in Section 5.

\section{Notation and Preliminaries}

Let $\Omega$ be an open connected set in $\mathbb{R}^d$ ($d$=2 or 3). We denote by the $L^2(\Omega)$ inner product and norm as $(\cdot,\cdot)$ and $\|\cdot\|$, respectively.  For other norms, we will label them with subscripts.  

We define the natural pressure and velocity function spaces for the NSE by 
	\begin{align*}
		&Q:=\{q\in {L}^2(\Omega): \int_{\Omega}q\ dx=0\} \quad \text{and}
		& X:=
		\{v\in H^1\left(\Omega\right): v=0~~\text{on}~ \partial\Omega\},
		\end{align*}
		along with the divergence-free velocity space
		\[
		 V:=
		\{v\in X:  (\nabla \cdot v,q)=0\ \forall q\in Q\}.
		\]
The dual space of $X$ is $H^{-1}(\Omega)$ with norm denoted $\| \cdot \|_{-1}$, and the dual space of $V$ is denoted by $V'$ with norm $\| \cdot \|_{V'}$.  We will use the notation $(\cdot,\cdot)$ to also represent dual pairings of $V$ and $X$ with their dual spaces.

We recall the following well known bounds and properties of the NSE nonlinearity form  \cite{Laytonbook}: \\
i) {\color{black}Skew-symmetry: $(v\cdot \nabla w,w)=0$} for all $v\in V$ and $w\in X$,\\
ii) there exists a constant $M$ dependent only on the size of $\Omega$ satisfying for all $v,w,\chi \in X$
\begin{align}
(v\cdot\nabla w,\chi) &\le M \| \nabla v \| \| \nabla w \| \| \nabla \chi \|, \label{bbounds} \\
\| v\cdot\nabla w \|_{V'} &\le M \| \nabla v \| \| \nabla w\|. \label{bbounds2}
\end{align}

The analysis of this paper is done primarily using the space $V$ and its dual space.  Results would be identical if we did the analysis using the discretely divergence-free subspace $V_h$ created from stable velocity-pressure spaces $(X_h,Q_h)\subset (X,Q)$.  The only difference would be that a 
skew-symmetric form of the nonlinear term would need to be used if the finite elements were not pointwise divergence free, e.g. 
\[
b^*(v,w,z) = (v\cdot\nabla w,z) + \frac12 ((\nabla \cdot v)w,z).
\]
One could also use other energy preserving formulations such as rotational and EMAC forms \cite{CHOR17}.

\subsection{Navier-Stokes preliminaries}
We consider the NSE on a domain $\Omega \subset \mathbb{R}^d (d=2,3)$ given by 
%\todo{This is a little bit overlap with \ref{NS1}. We can discuss after we receive the reviews.}
\begin{equation}\label{NS}
		\left\{\begin{aligned}
			-\nu \Delta u+u\cdot\nabla u+ \nabla p&={f} \quad \text{in}~\Omega,\\
			\nabla\cdot {u}&=0\quad \text{in}~\Omega,\\
			{u}&=0 \quad \text{on}~~\partial\Omega,
		\end{aligned}\right.
\end{equation}
where $u$ and $p$ are the unknown fluid velocity and pressure, respectively, $\nu>0$ is the kinematic viscosity,  and ${f}$ is a given external forcing (such as gravity or {\color{black} buoyancy}). The Reynolds number $Re \sim \frac{1}{\nu}$ is a physical constant that describes the complexity of a flow: higher $Re$ is generally associated with more complex physics and non-unique solutions.  For simplicity of analysis, we consider the system \eqref{NS} equipped with homogeneous Dirichlet boundary conditions, but our results are extendable to nonhomogeneous mixed Dirichlet/Neumann boundary conditions {\color{black} as well as} to solving the time dependent NSE at a fixed time step in a time stepping scheme.  

The weak form of \eqref{NS} is given by: Find $u\in V$ satisfying
\begin{equation}
\nu(\nabla u,\nabla v) + (u\cdot\nabla u,v) = (f,v)\ \ \forall v\in V.\label{NSw}
\end{equation}

It is well known that \eqref{NSw} admits weak solutions for any $\nu>0$ and $f\in H^{-1}(\Omega)$ \cite{Laytonbook}, and any such solution is bounded by 
\begin{equation}
\| \nabla u \| \le \nu^{-1} \| f \|_{-1}. \label{nsstab}
\end{equation}
Under a small data condition $\kappa:=M\nu^{-2} \|f \|_{H^{-1}}<1$, solutions are unique \cite{GR86,Laytonbook,temam}.  However, for $\kappa$ large enough, \eqref{NSw} can have multiple solutions \cite{Laytonbook}.

\subsection{Picard and nonlinear residual preliminaries}

The Picard iteration for solving \eqref{NSw} is defined by: given {\color{black} $v_k\in V$, find 
%$\tilde{u}_{k+1}\in V$ 
 $q(v_k)\in V$}
satisfying 
{\color{black}
\begin{equation}
\nu(\nabla {\color{black} q(v_k)},\nabla \chi) + (v_k \cdot\nabla {\color{black} q(v_k)},\chi) = ( f,\chi )\ \forall \chi \in V. \label{NSEPic}
\end{equation}}
It is shown in \cite{PRX19} that the solution operator associated with \eqref{NSEPic} is well-defined and is uniformly bounded:
{\color{black}
\begin{equation}
\| \nabla  q(v_k) \|  \le \nu^{-1} \| f\|_{-1}. \label{Picbound}
\end{equation} } 
Thus, we can define the fixed point Picard iteration as {\color{black} $v_{k+1}=q(v_k)$}, where $q:V\rightarrow V$ is the solution operator to \eqref{NSEPic}.

The fixed point residual satisfies \cite{GR86}
{\color{black}
\begin{equation}
 \| \nabla (q (v_{k})-v_{k} ) \| \le \kappa \| \nabla ( q(v_{k-1})-v_{k-1}) \|. \label{Picresid}
\end{equation}}

Note that if $\kappa<1$, \eqref{Picresid} indicates that the Picard iteration is contractive and will convergence linearly with rate (at least) $\kappa$ to the unique weak NSE solution.  However,  the Picard iteration for the NSE will diverge for $\kappa$ sufficiently large \cite{PR25}.  Our convergence results for NGMRES-Picard when $m=0$ assume only that $\kappa>0$, while for $m\ge 1$ we assume Picard is contractive, i.e. $\kappa<1$.  The difference comes from the additional higher order terms when $m\ge 1$.  In our numerical tests we use $\kappa \gg 1$, i.e. far above the range where Picard is contractive.

In addition to considering Picard residual convergence, we will also consider convergence using the NSE nonlinear residual, which is defined by $g:V \rightarrow V'$ which satisfies for all $v\in V$ and $\chi\in V$, 
\[
( g(v),\chi ) =   \nu (\nabla  v,\nabla \chi) + (v\cdot\nabla v,\chi) - (f,\chi).
\]
The $V'$ norm of $g(v)$ is thus defined by
\begin{equation}\label{eq:def-V'norm}
\|g(v)\|_{V'} = \max_{0\neq \chi \in V} \frac{(g(v),\chi)}{\|\nabla \chi\|}.
\end{equation}
Due to $v\in V$, the $v\cdot\nabla v$ term is not in $L^2(\Omega)$, and thus $V'$ is the appropriate norm to measure $g(v)$.

We now prove a preliminary lemma relating the Picard residual and nonlinear residual, as well as establishing a bound on the nonlinear residual for Picard.  We denote {\color{black}$\tilde v_{k+1}:=q(v_k)$, and} the Picard residual by
{\color{black}
\[
w_{k+1} = \tilde v_{k+1} - v_k = q(v_k) - v_k,
%w_{k+1} =  q(v_k) - v_k,
\]}
and the difference between two successive iterates by
\[
e_k = v_k - v_{k-1}.
\]

\begin{lemma}
The following bounds hold regarding the Picard residual and the NSE nonlinear residual:
\begin{align}
    \| \nabla w_{k+1} \| & \le \nu^{-1} \| g({\color{black}v}
_k) \|_{V'}, \label{wg1} \\
    \| g(\tilde {\color{black}v}_{k+1}) \|_{V'}  & \le \kappa \| g( {\color{black}v}
_{k}) \|_{V'}. \label{wg2}
\end{align}
\end{lemma}
\begin{proof}
Subtract $\nu(\nabla {\color{black}v}
_k,\nabla \chi) + ({\color{black}v}
_k \cdot \nabla {\color{black}v}
_k,\chi)$ from both sides of \eqref{NSEPic} to obtain for all $\chi\in V$,
\[
\nu(\nabla w_{k+1},\nabla \chi) + ({\color{black}v}
_k \cdot\nabla w_{k+1},\chi) = (f,\chi) + \nu (\nabla {\color{black}v}
_k,\nabla \chi) - ({\color{black}v}
_k \cdot \nabla {\color{black}v}
_k,\chi) = (-g({\color{black}v}
_k),\chi).
\]
Now choosing $\chi=w_{k+1}$ vanishes the nonlinear term and gives us
\[
\nu \| \nabla w_{k+1} \|^2 = -(g({\color{black}v}
_k),w_{k+1}) \le \| g({\color{black}v}
_k) \|_{V'} \|\nabla w_{k+1} \|,
\]
where the last step uses the definition of $V'$ norm, \eqref{eq:def-V'norm}, and which after simplifying completes the proof of the first bound of the lemma. 

For the second bound, expand $g(\tilde {\color{black}v}
_{k+1})$ to obtain for all $\chi\in V$,
\begin{align*}
(g(\tilde {\color{black}v}
_{k+1}),\chi) & = \nu(\nabla \tilde {\color{black}v}
_{k+1},\nabla \chi) + ({\color{black}v}
_k \cdot\nabla \tilde {\color{black}v}
_{k+1},\chi) - (f,\chi) \\ 
& = (w_{k+1} \cdot\nabla \tilde {\color{black}v}
_{k+1},\chi)
\end{align*}
with the last step thanks to \eqref{NSEPic}. Taking $V'$ norms of both sides and bounding the nonlinear term we get
\begin{align*}
\| g(\tilde {\color{black}v}
_{k+1}) \|_{V'} 
&\le M \| \nabla w_{k+1} \| \| \nabla \tilde {\color{black}v}
_{k+1} \| \\
&\le M \nu^{-1} \| f \|_{-1} \| \nabla w_{k+1} \|.
\end{align*}
Now combining with \eqref{wg1}, we obtain
\[
\| g(\tilde {\color{black}v}
_{k+1}) \|_{V'}  \le \kappa \| g( {\color{black}v}
_{k}) \|_{V'},
\]
which completes the proof of \eqref{wg2}.
\end{proof}

\subsection{NGMRES preliminaries}\label{sec:NGMRES-prelim} 
Before presenting NGMRES as an accelerate technique for the Picard iteration applied to the NSE, we first outline the NGMRES framework for accelerating a fixed-point iteration $q(u)$  used to solve $g(u)=0$.
 
\begin{alg}[Nonlinear GMRES with depth $m$] \label{alg:NGMRES}
 		\begin{algorithmic}[1]
 			\STATE  Input: $x_0$  and $m\geq0$ with $m\in\mathbb{N}_{\color{black} 0}$
 			\FOR {$k=0,1,\cdots$ until convergence}
 			\STATE Compute 
 			\begin{equation}\label{eq:xkp1-NG} 
 				u_{k+1} = q(u_k) + \sum_{i=0}^{m_k=\min\{k,m\}}\beta_i^{k+1} \left(q(u_k)- u_{k-i} \right),
 			\end{equation}
 			where  $\beta^{k+1}=\big(\beta_0^{k+1},\beta_1^{k+1},\cdots, \beta_{m_k}^{k+1}\big)$ is obtained by solving the  least-squares problem
 			\begin{equation}\label{eq:min-NG} 
 				\min_{\beta^{k+1}} \| g(q(u_k))+\sum_{i=0}^{m_k=\min\{k,m\}} \beta_i^{k+1} \left(g(q(u_k))-g(u_{k-i}) \right)\|^2.
 			\end{equation}
 			\ENDFOR
            \STATE  Output: $u_{k+1}$
 		\end{algorithmic}	 
 	\end{alg}
Note that in \cref{alg:NGMRES}, we formulate the least‑squares problem \eqref{eq:min-NG} in a non-constrained form, and in practice the $\ell^2$ norm is generally used. However, one may use the constrained formulation, as we did in Section \ref{sec: cs-ngmres}, for the sake of simplifying the convergence analysis for NSE. $m$ is a given integer. If $m_k=k$ for all $k$, we refer to the corresponding method as full NGMRES.  When  \cref{alg:NGMRES} is applied to solve the linear system $Au=b$ with $q(u_k)=u_k+(b-Au_k)$,  \cite{greif2026convergence} proves that if GMRES and full NGMRES use the same initial guess and the GMRES residual norm decreases strictly at every iteration, then the iterates generated by GMRES and full NGMRES coincide. Similar results have been extended to preconditioned fixed-point iterations, establishing an equivalence between full NGMRES and preconditioned GMRES under analogous assumptions; see \cite{he2026ngmresprecon}.

For the least-squares problem \eqref{eq:min-NG}, to the best of our knowledge, only the $\ell^2$ norm has been used in the literature. It is therefore natural to ask which norm is most appropriate--either for achieving the best performance or for ensuring that the formulation remains meaningful in the PDE setting. Although NGMRES has demonstrated its efficiency, its application and convergence analysis remain limited, particularly for nonlinear problems. In the next section, we apply NGMRES to accelerate Picard iteration for NSE and provide a convergence analysis of NGMRES in this setting. 

\section{Convergence analysis of NGMRES for NSE}\label{sec: cs-ngmres}

We prove in this section a convergence estimate for NGMRES-Picard which uses the nonlinear residual $g$ in its associated optimization problem.  We note that using instead the Picard residual (i.e., $w_{k+1}(q(u_k))=q(q(u_k))-q(u_k)$) for the optimization problem would create an extra Picard solve at each nonlinear iteration, making NGMRES-Picard twice as expensive at each iteration compared to Picard. Note that in \cite{he2026ngmresprecon}, the use of the underlying fixed-point iteration residual in the optimization problem is proposed. One potential advantage is that the $\ell_2$ norm may be more suitable for the NSE. We leave a detailed investigation of this for future work.

None of the main results require finite dimensionality.  For simplicity we work in $V$ but the same proofs will work in the $V_h$ setting provided the nonlinear term is skew-symmetrized (e.g. through skew-symmetric form, rotational form, or EMAC form \cite{CHOR17,GS98,LMNOR09}).  However, in Section 3.3 where the optimization norm implementation is discussed and in particular whether it is okay to use $\ell^2$ (yes in 2D, no in 3D), the finite dimensional case is considered.

In \cref{ngmresalg}, we present NGMRES to accelerate the Picard iteration for the NSE, where we choose the norm in \eqref{eq:min-NG} as the $V'$ norm defined in \eqref{eq:def-V'norm},  use the constrained form for the least-squares problem in \cref{alg:NGMRES}.

\begin{alg}[NGMRES for NSE] \label{ngmresalg} \ \\
NGMRES for the NSE is defined at step $k+1$ by: \\

\begin{enumerate}
\item Find $\tilde u_{k+1} = q(u_k)$ from \eqref{NSEPic}. 
\item Find the coefficients $(\alpha^{k+1}_{k+1},\alpha^{k+1}_{k},...,\alpha^{k+1}_{k-m})$ that minimize
\[
\min_{\sum_{i={k-m}}^{k+1}\alpha^{k+1}_i=1 } \| \alpha^{k+1}_{k+1} g(\tilde u_{k+1}) + \alpha^{k+1}_k g(u_k)  + ... \alpha^{k+1}_{k-m}g(u_{k-m}) \|_{V'}.
\]
\item Set 
\[
u_{k+1} = \alpha^{k+1}_{k+1} \tilde u_{k+1} + \alpha^{k+1}_k u_k +\alpha^{k+1}_{k-1} u_{k-1} + ... +  \alpha^{k+1}_{k-m} u_{k-m}. 
\]
\end{enumerate}
\end{alg}
The optimization problem plays a crucial role in convergence. We define $\theta_{k+1}$ to be the gain of the optimization problem over the previous iterate and $\gamma_{k+1}$ to be the gain of the optimization problem over Picard, in the following sense:
\begin{align*}
 \theta_{k+1}&:= \frac{ \| \alpha^{k+1}_{k+1} g(\tilde u_{k+1}) + \alpha^{k+1}_k g(u_k) + ... + \alpha^{k+1}_{k-m}g(u_{k-m}) \|_{V'}  }{ \| g(u_k) \|_{V'} },\\
 \gamma_{k+1}&:= \frac{ \| \alpha^{k+1}_{k+1} g(\tilde u_{k+1}) + \alpha^{k+1}_k g(u_k) + ... + \alpha^{k+1}_{k-m}g(u_{k-m}) \|_{V'}  }{ \| g(\tilde u_{k+1}) \|_{V'} }.
 \end{align*}
Note that $\theta_{k+1},\gamma_{k+1}\le 1$ and can only reach 1 if the minimum is achieved when $\alpha^{k+1}=(0,1,0,...,0)$ and $(1,0,...,0)$, respectively.  
We will assume the optimization coefficients are uniformly bounded, i.e. $\max_{j,i} |\alpha_j^i | \le \bar\alpha$. This assumption is equivalent to assuming $\{  g(\tilde u_{k+1}),\  g(u_{k}),...,\  g(u_{k-m})\}$ are linearly independent with respect to the $V'$ inner product, {\color{black} i.e. the least squares problem is well-posed (its matrix has full column rank).  If they are close to $V'$-linearly dependent, then in practice either $m$ can reduced until the least squares system is well-posed, or filtering can be done analogous to what \cite{PR23} does for AA.}

We denote the Picard residual by
\[
w_{k+1} = \tilde u_{k+1} - u_k = q(u_k) - u_k,
\]
and the difference between two successive iterates by
\[
e_k = u_k - u_{k-1}.
\]

\subsection{Convergence of NGMRES-Picard for the base case of $m=0$} \ \\
We now prove a one-step convergence theorem for NGMRES-Picard with $m=0$ for the NSE.  The base case has a simpler convergence proof than the general case, so we believe it is instructive for the reader to do the base case first.
NGMRES-Picard in the base case $m=0$ is defined at step $k+1$ by: \\

\begin{enumerate}
\item Find $\tilde u_{k+1} = q(u_k)$ from \eqref{NSEPic}.
\item Find
\[
\alpha^{k+1} = \arg \min_{\alpha} \| \alpha g(\tilde u_{k+1}) + (1-\alpha) g(u_k) \|_{V'}.
\]
\item Set 
\[
u_{k+1} = \alpha^{k+1} \tilde u_{k+1} + (1-\alpha^{k+1}) u_k.
\]
\end{enumerate}
We will assume that $\max_k | \alpha^k | \le \bar \alpha$ {\color{black} for some $\bar\alpha>0$.}

\begin{theorem} \label{thm0}
Assume that $\max_k | \alpha^k | \le \bar \alpha$. For any $\kappa>0$,  the nonlinear residual for NGMRES-Picard with $m=0$ at step $k+1$ satisfies the bounds
\begin{align*}
\| g(u_{k+1}) \|_{V'}  &\le \theta_{k+1} \| g(u_k) \|_{V'}
+ M \nu^{-2} \bar\alpha(1+\bar\alpha)  \| g(u_k) \|_{V'}^2, \\
\| g(u_{k+1}) \|_{V'}  &\le \gamma_{k+1}\kappa \| g(u_k) \|_{V'}
+ M \nu^{-2} \bar\alpha(1+\bar\alpha)  \| g(u_k) \|_{V'}^2, 
\end{align*}
\end{theorem}
where
\begin{align*}
\theta_{k+1}&=\frac{\|\alpha^{k+1} g(\tilde u_{k+1}) + (1-\alpha^{k+1}) g(u_k) \|_{V'}}{\| g(u_k) \|_{V'}},\\
\gamma_{k+1}&=\frac{\|\alpha^{k+1} g(\tilde u_{k+1}) + (1-\alpha^{k+1}) g(u_k) \|_{V'}}{\| g(\tilde u_{k+1}) \|_{V'}}.
\end{align*}

\begin{remark}
The first bound shows the sequence defined by $\{ \| g(u_j)\|_{V'} \}_j$ is locally contractive, with contraction number at each iteration given $\theta_{k+1}$.  The term $\gamma_{k+1}$ represents the acceleration provided by NGMRES-Picard over Picard.
\end{remark}
{\color{black}
\begin{remark}
The higher order terms for the $m=0$ (depth 1) NGMRES-Picard result of the theorem are  quadratic.  This result may indicate that NGMRES is more compatible with general superlinearly convergent iterations compared to AA, which has 
a residual bounds with less than quadratic higher order terms and is known to slow superlinear convergence \cite{PR25,RX23}.  The authors plan to study NGMRES for superlinears iteration in future work.
\end{remark}

}
\begin{proof}
We begin by expanding $g(u_{k+1})$ and adding and subtracting $u_k$ to the first argument of the nonlinear term to get for all $v\in V$,
\begin{align*}
(g(u_{k+1}),v) & = \nu(\nabla u_{k+1},\nabla v) + (u_{k+1}\cdot\nabla u_{k+1},v) - (f,v) \\
& = \nu(\nabla u_{k+1},\nabla v) + (u_{k}\cdot\nabla u_{k+1},v) - (f,v) {\color{black} + (} (u_{k+1}-u_{k})\cdot\nabla u_{k+1},v).
\end{align*}
Next, expand $u_{k+1}=\alpha^{k+1} \tilde u_{k+1} + (1-\alpha^{k+1}) u_k$ in the first and second right hand side terms above to obtain
\begin{align*}
{\color{black}(g(u_{k+1}),v)} 
%& = \alpha^{k+1} \left(  -\nu\Delta \tilde u_{k+1} + u_{k}\cdot\nabla \tilde u_{k+1} - f \right)
%\\ & \ \ \ + (1-\alpha^{k+1}) \left( -\nu\Delta  u_{k} + u_{k}\cdot\nabla  u_{k} - f \right)+ (u_{k+1}-u_{k})\cdot\nabla u_{k+1}  \\
& = \left( (1-\alpha^{k+1}) g(u_k) + (u_{k+1}-u_{k})\cdot\nabla u_{k+1},v\right), 
\end{align*}
where the last step uses the definition of $\tilde u_{k+1}$ (see \eqref{NSEPic}) to vanish the first right hand side term  and the definition of $g(u_k)$.  Next we add and subtract $\alpha^{k+1} g(\tilde u_{k+1})$ to the right hand side, which produces
\begin{multline}    
(g(u_{k+1}),v) = \left( \alpha^{k+1} g(\tilde u_{k+1}) + (1-\alpha^{k+1}) g(u_k),v \right) \\ + ((u_{k+1}-u_{k})\cdot\nabla u_{k+1},v) -  (\alpha^{k+1} g(\tilde u_{k+1}),v). \label{ident1}
\end{multline}

We now rewrite $g(\tilde u_{k+1})$ as
\begin{align*}
  (g(\tilde u_{k+1}),v)
  = & \nu(\nabla \tilde u_{k+1},\nabla v) + (\tilde u_{k+1} \cdot\nabla \tilde u_{k+1},v) - (f,v) \\
  = & \nu(\nabla  \tilde u_{k+1},\nabla v) + (u_{k} \cdot\nabla \tilde u_{k+1},v) - (f,v)  + ( (\tilde u_{k+1}-u_k) \cdot\nabla \tilde u_{k+1},v) \\
 = & ( (\tilde u_{k+1}-u_k) \cdot\nabla \tilde u_{k+1},v)
\end{align*}
for all $v\in V$, where the last step uses the definition of $\tilde u_{k+1}$; see \eqref{NSEPic}.

Again using $u_{k+1}=\alpha^{k+1} \tilde u_{k+1} + (1-\alpha^{k+1}) u_k$ to simplify the last two terms in \eqref{ident1}  produces
\begin{align*}
&( (u_{k+1} -u_{k})\cdot\nabla u_{k+1},v) -  \alpha^{k+1} (g(\tilde u_{k+1}),v) \\
 = & ( (\alpha^{k+1} \tilde u_{k+1} -\alpha^{k+1}u_{k})\cdot\nabla u_{k+1},v)
- \alpha^{k+1} ( (\tilde u_{k+1}-u_k) \cdot\nabla \tilde u_{k+1},v)  \\
 = & \alpha^{k+1}  ( (\tilde u_{k+1}-u_k) \cdot\nabla  (u_{k+1} - \tilde u_{k+1} ),v) \\
 =&  \alpha^{k+1} ( w_{k+1} \cdot\nabla  (\alpha^{k+1} \tilde u_{k+1} + (1-\alpha^{k+1}) u_k - \tilde u_{k+1} ),v) \\
  =& -\alpha^{k+1} (1-\alpha^{k+1}) ( w_{k+1} \cdot\nabla w_{k+1},v) .
\end{align*}
We have thus established the identity
\begin{multline}
(g(u_{k+1}){\color{black},v)}  \\
  (  \left( \alpha^{k+1} g(\tilde u_{k+1})  + (1-\alpha^{k+1}) g(u_k) \right) -\alpha^{k+1} (1-\alpha^{k+1}) w_{k+1} \cdot\nabla w_{k+1}{\color{black},v)}. \label{gwidentity}
\end{multline}
Taking the $\| \cdot \|_{V'}$ norm of both sides, using \eqref{bbounds2}, $| \alpha^{k+1} |\le \bar\alpha$ and reducing yields
\begin{multline}
\| g(u_{k+1}) \|_{V'} \\ \le \|  \alpha^{k+1} g(\tilde u_{k+1}) + (1-\alpha^{k+1}) g(u_k) \|_{V'}
+ {\color{black} M} \bar \alpha (\bar \alpha+1) \cdot \| \nabla w_{k+1} \|^2. \label{m0A1}
\end{multline}
Now using \eqref{wg1} and the definition of $\theta_{k+1}$ in \eqref{m0A1} gives us
\begin{align*}
\| g(u_{k+1}) \|_{V'} 
& \le \theta_{k+1} \| g(u_k) \|_{V'} 
+ M \nu^{-2}  \bar \alpha (\bar \alpha+1)\cdot \| g(u_k) \|_{V'}^2,
\end{align*}
and similarly using \eqref{wg1}, the definition of $\gamma_{k+1}$ together with \eqref{wg2} provides with \eqref{m0A1} 
\begin{align*}
\| g(u_{k+1}) \|_{V'} 
& \le \gamma_{k+1}\kappa \| g(u_k) \|_{V'} 
+ M \nu^{-2}  \bar \alpha (\bar \alpha+1)\cdot \| g(u_k) \|_{V'}^2.
\end{align*}
This completes the proof.
\end{proof}

\subsection{Convergence analysis of NGMRES for NSE for general  $m$}

The convergence proof for general $m$ has the same framework as the $m=0$ case, but is more complex because $u_{k+1}$ is defined as a combination of  $m+2$ terms.  Additionally, to reduce notation, since all optimization coefficients take the form $\alpha^{k+1}_j$ (that is, they all have the $k+1$ superscript), we drop the superscripts in this subsection.

We also prove a preliminary bound for the difference between successive iterates.
\begin{lemma}
Assuming $\kappa<1$, then for any depth $m$, the difference {\color{black} $e_k = u_k - u_{k-1}$} between successive NGMRES for NSE iterates satisfies
\begin{equation}
\| \nabla e_k \| \le \frac{\nu^{-1}}{1-\kappa} \left(  \| g(u_k) \|_{V'} +  \| g(u_{k-1}) \|_{V'} \right),\label{eg1}
\end{equation}
\end{lemma}
\begin{proof}
Writing $e_k = u_k - u_{k-1} = u_k - \tilde u_k + \tilde u_k - u_{k-1}$, we have that
\[
e_k = (u_k - \tilde u_k) + w_k,
\]
and hence {\color{black} by the triangle inequality,
\begin{equation}
\| \nabla e_k \| \le \| \nabla (u_k - \tilde u_k) \| + \| \nabla w_k \|. \label{ebound}
\end{equation}
Since $w_k$ is a Picard residual, we have by \eqref{wg1} that
\begin{equation}
    \| \nabla w_k \| \le \nu^{-1} \| g(u_{k-1})\|_{V'}. \label{ebound1a}
\end{equation}
Thus to bound $\| \nabla e_k\|$, it remains only to bound $\| \nabla (u_k - \tilde u_k) \|$.
  Using  the definition of $g(u_k)$ and \eqref{NSEPic} (with input $u_{k-1}$ and solution $\tilde u_k = q(u_{k-1})$ gives the two equations}, for all $v\in V$,
\begin{align*}
\nu(\nabla u_k,\nabla v) + (u_k\cdot\nabla u_k,v)  & = (f,v) + (g(u_k),v), \\
\nu(\nabla \tilde u_k,\nabla v) + (u_{k-1}\cdot\nabla \tilde u_k,v)  & = (f,v). 
\end{align*}
Subtracting the above equations gives {\color{black} for all $v\in V$,}
\[
\nu(\nabla (u_k - \tilde u_k),\nabla v) + (u_k\cdot\nabla (u_k - \tilde u_k),v) + (e_k \cdot\nabla \tilde u_k,v) = (g(u_k),v).
\]
Choosing $v=(u_k - \tilde u_k)$ vanishes the first nonlinear term {\color{black} due to skew-symmetry since $u_{k-1}\in V$,} and gives
\begin{align*}
\nu \| \nabla (u_k - \tilde u_k) \|^2 & = - (e_k \cdot\nabla \tilde u_k,u_k - \tilde u_k) + (g(u_k),u_k - \tilde u_k) \\
& \le M \| \nabla e_k \| \| \nabla \tilde u_k \| \| \nabla (u_k - \tilde u_k) \| + \| g(u_k) \|_{V'}\| \nabla (u_k - \tilde u_k) \|,
\end{align*}
{\color{black}with the last step using Lemma \ref{bbounds}.  After dividing by sides
by $\nu \| \nabla (u_k - \tilde u_k) \|$,} using \eqref{Picbound} {\color{black} to bound $\| \nabla \tilde u_k \|\le \nu^{-1} \| f \|_{-1}$, and recalling $\kappa=M\nu^{-2} \| f\|_{-1}$,} this reduces to 
\[
\| \nabla (u_k - \tilde u_k) \| \le \kappa \| \nabla e_k \| + \nu^{-1}  \| g(u_k) \|_{V'}.
\]
Combining this bound {\color{black} with \eqref{ebound} and \eqref{ebound1a}}, we get the estimate
%\[
%(1-\kappa) \| \nabla e_k \| \le \| \nabla w_k \| +  \nu^{-1}  \| g(u_k) \|_{V'}.
%\]
%Finally, using \eqref{wg1}, we get
\[
\| \nabla e_k \| \le \frac{\nu^{-1}}{1-\kappa} \left(  \| g(u_k) \|_{V'} +  \| g(u_{k-1}) \|_{V'} \right),
\]
which finishes the proof.
\end{proof}

We begin with a lemma for an important $g(u_{k+1})$ identity.

\begin{lemma}\label{glemma}
It holds for all $\chi\in V$ that
\begin{equation}
(g(u_{k+1}),\chi) = \left( \alpha_{k+1}g(\tilde u_{k+1}) + \alpha_{k}g(u_{k}) + ... +  \alpha_{k-m}g(u_{k-m}),\chi \right) + (R,\chi),\label{Glem}
\end{equation}
where
\begin{align*}
(R,\chi)  = & 
\bigg( -(1-\alpha_{k+1}) 
\left(  \alpha_{k+1} w_{k+1} - \alpha_{k-1} e_k -  ... -   \alpha_{k-m} \sum_{j=k-m+1}^k e_j \right) \cdot\nabla w_{k+1} \\
& -\alpha_{k-1} \left(  \alpha_{k+1} w_{k+1} -  \alpha_{k-1} e_k -  ... -   \alpha_{k-m}\sum_{j=k-m+1}^k e_j \right) \cdot\nabla e_{k} 
 - ... \\
& - \alpha_{k-m}\left(  \alpha_{k+1} w_{k+1} -\alpha_{k-1} e_k - ...-  \alpha_{k-m} \sum_{j=k-m+1}^k e_j \right) \\
& \ \ \ \cdot\nabla(e_k + e_{k-1} + ... + e_{k-m+1}) \\
& -\left( \alpha_{k-1} e_k + ... +\alpha_{k-m}\sum_{j=k-m+1}^k e_j \right) \cdot\nabla w_{k+1} \\
&- \left( \alpha_{k-1} e_k + ... +\alpha_{k-m} \sum_{j=k-m+1}^k e_j\right) \cdot\nabla e_k \\
& -  \left( \alpha_{k-2} (e_k + e_{k-1}) + ... +\alpha_{k-m}\sum_{j=k-m+1}^k e_j \right) \cdot\nabla e_{k-1}  
 - ... \\
 &-\alpha_{k-m} \sum_{j=k-m+1}^k e_j\cdot\nabla e_{k-m+1},\chi \bigg).
\end{align*}
\end{lemma}
\begin{proof}
The proof begins with an identity for $u_{k+1} - \tilde u_{k+1}$, which is constructed by expanding $u_{k+1}$ using its definition to obtain
\[
u_{k+1} - \tilde u_{k+1}  = (\alpha_{k+1}-1) \tilde u_{k+1} + \alpha_k u_k + \alpha_{k-1}u_{k-1} + ... + \alpha_{k-m} u_{k-m}.
\]
Writing $\alpha_k = 1-\alpha_{k+1} - \alpha_{k-1} - ... - \alpha_{k-m}$, we have that
\begin{align}
u_{k+1} - \tilde u_{k+1}  & = (\alpha_{k+1}-1) (\tilde u_{k+1} -u_k)  + \alpha_{k-1}(u_{k-1}-u_k) + ... + \alpha_{k-m} (u_{k-m}-u_k) \nonumber \\
& = -(1-\alpha_{k+1}) w_{k+1} - \alpha_{k-1} e_k - ... - \alpha_{k-m}(e_k + e_{k-1} + ... + e_{k-m+1}). \label{GI0}
\end{align}

Next, we expand $g(u_{k+1})$ by using the definition of $g$, then use the definition of $u_{k+1}$ and finally regrouping terms and using \eqref{NSEPic} to get for all $\chi\in V$,
\begin{align*}
(g(u_{k+1}),\chi) 
& = \nu(\nabla  u_{k+1},\nabla \chi) + (u_{k+1}\cdot\nabla u_{k+1},\chi) - ( f ,\chi) \\
&= \alpha_{k+1}\left( \nu(\nabla \tilde u_{k+1},\nabla \chi) + (u_{k+1}\cdot\nabla \tilde u_{k+1},\chi) - (f,\chi) \right)  \\
& \ \ \ + \alpha_k \left( \nu(\nabla u_{k},\nabla \chi) + (u_{k+1}\cdot\nabla u_{k},\chi) - (f,\chi) \right) \\
& \ \ \ + ... + \alpha_{k-m} \left( \nu (\nabla  u_{k-m},\nabla \chi) + (u_{k+1}\cdot\nabla u_{k-m} ,\chi) - (f,\chi) \right) \\
& = \left( \alpha_{k+1}g(\tilde u_{k+1}) + \alpha_{k}g(u_{k}) + ... +  \alpha_{k-m}g(u_{k-m}),\chi \right)-  \alpha_{k+1}(g(\tilde u_{k+1}),\chi) \\
& \ \ \ + \alpha_{k+1} ((u_{k+1}-u_k)\cdot\nabla \tilde u_{k+1} ,\chi)
+ \alpha_k ( (u_{k+1}-u_k)\cdot\nabla u_{k},\chi) \\
& \ \ \ + ...
+ \alpha_{k-m} ( (u_{k+1}-u_{k-m})\cdot\nabla u_{k-m},\chi) ,
\end{align*}
where in the last step we added and subtracted $\alpha_{k+1}g(\tilde u_{k+1})$.  Since by \eqref{NSEPic}
\[
(g(\tilde u_{k+1}),\chi) = (\nu(\nabla \tilde u_{k+1},\nabla \chi) + (\tilde u_{k+1}\cdot\nabla \tilde u_{k+1},\chi) - (f,\chi)) = ( (\tilde u_{k+1} - u_k)\cdot \nabla \tilde u_{k+1},\chi),
\]
we have that
\begin{equation}
(g(u_{k+1}),\chi) = \left( \alpha_{k+1}g(\tilde u_{k+1}) + \alpha_{k}g(u_{k}) + ... +  \alpha_{k-m}g(u_{k-m}),\chi \right) + (R,\chi),\label{GI1}
\end{equation}
where
\begin{multline*}
(R,\chi) =  \alpha_{k+1} ((u_{k+1}-u_k)\cdot\nabla \tilde u_{k+1},\chi) 
+ \alpha_k ((u_{k+1}-u_k)\cdot\nabla u_{k},\chi)\\
+ ...
+ \alpha_{k-m} ( (u_{k+1}-u_{k-m})\cdot\nabla u_{k-m},\chi) - \alpha_{k+1} ((\tilde u_{k+1} - u_k)\cdot \nabla \tilde u_{k+1},\chi).
\end{multline*}
Next, we combine terms in $R$ (except the last one) by using the definition of $u_{k+1}$ and adding and subtracting appropriately via
\begin{align*}
&\big( \alpha_{k+1}  (u_{k+1}-u_k)\cdot\nabla \tilde u_{k+1} 
+ \alpha_k (u_{k+1}-u_k)\cdot\nabla u_{k}
+ ...
\\ & \ \ \ \ \ + \alpha_{k-m} (u_{k+1}-u_{k-m})\cdot\nabla  u_{k-m} ,\chi \big)\\
 =& \big( \alpha_{k+1}  (u_{k+1}-u_k)\cdot\nabla \tilde u_{k+1} 
+ \alpha_k (u_{k+1}-u_k)\cdot\nabla u_{k}
+ ...
\\
& \ \ \ \ \ + \alpha_{k-m} (u_{k+1}-u_{k})\cdot\nabla  u_{k-m} \\
& \ \ 
+ \alpha_{k-1} (u_{k}-u_{k-1})\cdot\nabla  u_{k-1}
+ ...
+ \alpha_{k-m} (u_{k}-u_{k-m})\cdot\nabla  u_{k-m} ,\chi \big) \\
 =& \big( (u_{k+1} - u_{k}) \cdot\nabla u_{k+1}  + \alpha_{k-1} (u_{k}-u_{k-1})\cdot\nabla  u_{k-1}
+ ...\\
&\ \  + \alpha_{k-m} (u_{k}-u_{k-m})\cdot\nabla  u_{k-m},\chi \big).
\end{align*}
Now by expanding the first term of the last line, we can write
\begin{align*}
(R,\chi) &= \big( ( \alpha_{k+1}\tilde u_{k+1} + (\alpha_k-1) u_k + ... + \alpha_{k-m}u_{k-m} ) \cdot\nabla u_{k+1} \\  
  & \ \ \ + \alpha_{k-1} (u_{k}-u_{k-1})\cdot\nabla  u_{k-2}+\alpha_{k-2} (u_{k}-u_{k-2})\cdot\nabla  u_{k-2}
 + ...\\
& \ \ \ + \alpha_{k-m} (u_{k}-u_{k-m})\cdot\nabla  u_{k-m}
- \alpha_{k+1} (\tilde u_{k+1} - u_k)\cdot \nabla \tilde u_{k+1},\chi \big),
\end{align*}
and since $\alpha_k - 1 = -\alpha_{k+1} - \alpha_{k-1} - ... - \alpha_{k-m}$, 
\begin{align*}
(R,\chi) =& \big( ( \alpha_{k+1}(\tilde u_{k+1}-u_k) + \alpha_{k-1}(u_{k-1}-u_k)+ ... + \alpha_{k-m}(u_{k-m}-u_k )) \cdot\nabla u_{k+1} \\
&+ \alpha_{k-1} (u_{k}-u_{k-1})\cdot\nabla  u_{k-1}
+ ...\\
&+ \alpha_{k-m} (u_{k}-u_{k-m})\cdot\nabla  u_{k-m}
- \alpha_{k+1} (\tilde u_{k+1} - u_k)\cdot \nabla \tilde u_{k+1} ,\chi \big).
\end{align*}
Combining the first and last terms now gives us
\begin{multline*}
(R,\chi) = \\ \big( \alpha_{k+1} w_{k+1}\cdot\nabla (u_{k+1} - \tilde u_{k+1}) - (\alpha_{k-1}e_k + ... + \alpha_{k-m}(e_k + e_{k-1} ... + e_{k-m+1}) ) \cdot\nabla u_{k+1} \\ 
+ \alpha_{k-1} e_k \cdot\nabla  u_{k-1}
+ ...
+ \alpha_{k-m} (e_k + e_{k-1} + ... + e_{k-m+1})\cdot\nabla  u_{k-m},\chi \big),
\end{multline*}
which after regrouping terms becomes
\begin{multline*}
(R,\chi) = \big( \alpha_{k+1} w_{k+1}\cdot\nabla (u_{k+1} - \tilde u_{k+1}) 
 - \alpha_{k-1} e_k \cdot\nabla (u_{k+1} - u_{k-1}) \\
 - ... -  \alpha_{k-m} (e_k + e_{k-1}  + ... + e_{k-m+1})\cdot\nabla (u_{k+1} - u_{k-m}),\chi \big).
\end{multline*}
We now add and subtract to the $(u_{k+1} - u_{k-j})$ terms via
\[
u_{k+1} - u_{k-j} = (u_{k+1} - \tilde u_{k+1}) + w_{k+1} + e_k + ... + e_{k-j+1},
\]
which implies
\begin{multline*}
(R,\chi) = \\
\big( \alpha_{k+1} w_{k+1}\cdot\nabla (u_{k+1} - \tilde u_{k+1}) 
 - \alpha_{k-1} e_k \cdot\nabla \left( (u_{k+1} - \tilde u_{k+1}) + w_{k+1} + e_k \right)
 - ... \\ - \alpha_{k-m} (e_k + e_{k-1}  + ... + e_{k-m+1})\cdot\nabla \left( (u_{k+1} - \tilde u_{k+1}) + w_{k+1} + e_k + ... + e_{k-m+1} \right) ,\chi \big),
\end{multline*}
and thus
\begin{multline*}
(R,\chi) = \\ \bigg( \left(  \alpha_{k+1} w_{k+1} -\alpha_{k-1} e_k - ... - \alpha_{k-m} (e_k + e_{k-1}  + ... + e_{k-m+1})\right) \cdot\nabla (u_{k+1} - \tilde u_{k+1}) \\
-\left( \alpha_{k-1} e_k + ... +\alpha_{k-m} (e_k + e_{k-1}  + ... + e_{k-m+1}) \right) \cdot\nabla w_{k+1} \\
- \left( \alpha_{k-1} e_k + ... +\alpha_{k-m} (e_k + e_{k-1}  + ... + e_{k-m+1}) \right) \cdot\nabla e_k \\
- \left( \alpha_{k-2} (e_k + e_{k-1}) + ... +\alpha_{k-m} (e_k + e_{k-1}  + ... + e_{k-m+1}) \right) \cdot\nabla e_{k-1} \\ 
- ...  -\alpha_{k-m} (e_k + e_{k-1}  + ... + e_{k-m+1})\cdot\nabla e_{k-m+1},\chi \bigg).
\end{multline*}
Finally, using \eqref{GI0}, we have that
\begin{align*}
&(R,\chi)  =  \bigg(  \\
& -(1-\alpha_{k+1}) \left(  \alpha_{k+1} w_{k+1} - \alpha_{k-1} e_k - ... -  \alpha_{k-m} (e_k + e_{k-1}  + ... + e_{k-m+1})\right) \cdot\nabla w_{k+1} \\
& -\alpha_{k-1} \left(  \alpha_{k+1} w_{k+1} - \alpha_{k-1} e_k - ... - \alpha_{k-m} (e_k + e_{k-1}  + ... + e_{k-m+1})\right) \cdot\nabla e_{k} \\
& - ... \\
& - \alpha_{k-m}\left(  \alpha_{k+1} w_{k+1} + \alpha_{k-1} e_k + ... +  \alpha_{k-m} (e_k + e_{k-1}  + ... + e_{k-m+1})\right) \\
& \ \ \ \ \cdot\nabla(e_k + e_{k-1} + ... + e_{k-m+1}) \\
& - \left( \alpha_{k-1} e_k + ... +\alpha_{k-m} (e_k + e_{k-1}  + ... + e_{k-m+1}) \right) \cdot\nabla w_{k+1} \\
&-  \left( \alpha_{k-1} e_k + ... +\alpha_{k-m} (e_k + e_{k-1}  + ... + e_{k-m+1}) \right) \cdot\nabla e_k \\
& -  \left( \alpha_{k-2} (e_k + e_{k-1}) + ... +\alpha_{k-m} (e_k + e_{k-1}  + ... + e_{k-m+1}) \right) \cdot\nabla e_{k-1} \\ 
& - ...  - \alpha_{k-m} (e_k + e_{k-1}  + ... + e_{k-m+1})\cdot\nabla e_{k-m+1},\chi \bigg).
\end{align*}

\end{proof} 

We can now prove a convergence result for NGMRES with general depth $m$.
\begin{theorem}\label{thmm}
Assuming $\kappa<1$ 
and $\max_{k,j} | \alpha^k_j | \le \bar \alpha$, for any $m\ge 1$ we have that
\begin{align*}
    \| g(u_{k+1}) \|_{V'} \le \theta_{k+1} \| g(u_{k})   \|_{V'} + \mathcal{O}\left( \| g(u_{k})   \|_{V'}^2 + \| g(u_{k-1})   \|_{V'}^2 + ... + \| g(u_{k-m})   \|_{V'}^2\right), \\
    \| g(u_{k+1}) \|_{V'} \le \gamma_{k+1} \kappa \| g(u_{k})   \|_{V'} + \mathcal{O}\left( \| g(u_{k})   \|_{V'}^2 + \| g(u_{k-1})   \|_{V'}^2 + ... + \| g(u_{k-m})   \|_{V'}^2\right).
\end{align*}

\end{theorem}
\begin{remark}
{\color{black} The assumption $\kappa<1$ is needed in the proof to bound terms of the form $\| \nabla e_j \|$ using \eqref{eg1}.  For the $m=0$ result, no such terms arise due to the extrapolation used by NGMRES.   The condition $\kappa<1$ is well known as a sufficient condition for steady NSE well-posedness and for $\kappa$ large enough it is known that NSE solutions need not be unique \cite{G94,Laytonbook}, and also the condition was assumed in the AA-Picard proof of acceleration in \cite{PRX19}.  Our numerical tests use $\kappa<1$ but also $\kappa \gg 1$, and NGMRES enables convergence in many of the $\kappa\gg 1$ tests.  Hence we believe that a convergence result similar to Theorem \ref{thmm} is likely true for $\kappa>1$, perhaps with worse bounds in the higher order terms, but we have not (yet) discovered another way to bound the $\| \nabla e_j \|$ terms without the $\kappa<1$ assumption.}
\end{remark}

\begin{remark}
The same arguments from the $m=0$ case apply here regarding using $\ell^2$ or $L^2$ as the optimization norm instead of $V'$.  That is, in 2D this is fine but in 3D it is generally not a good choice. {\color{black}Further discussion on the effects of using these optimization norms can be found below in Section 3.3.}
\end{remark}

\begin{proof}
Taking $V'$ norms of both sides in \eqref{Glem}, we obtain
\[
\| g(u_{k+1}) \|_{V'} \le \| \alpha_{k+1}g(\tilde u_{k+1}) + \alpha_{k}g(u_{k}) + ... +  \alpha_{k-m}g(u_{k-m})  \|_{V'} + \| R \|_{V'},
\]
and thus using the definition of $\theta_{k+1}$ we have that
\[
\| g(u_{k+1}) \|_{V'} \le \theta_{k+1} \| g(u_{k})   \|_{V'} + \| R \|_{V'},
\]
and similarly using \eqref{wg2} and the definition of $\gamma_{k+1}$,
\[
\| g(u_{k+1}) \|_{V'} \le \gamma_{k+1} \kappa \| g(u_{k})   \|_{V'} + \| R \|_{V'},
\]
Thus it remains to bound $R$.  We first note that each term of $R$ takes one the forms: $w_{k+1}\cdot\nabla w_{k+1}$, $e_j\cdot\nabla w_{k+1}$ ( $k-m+1 \le j \le k$), $w_{k+1} \cdot\nabla e_j$ ( $k-m+1 \le j \le k$), or $e_j\cdot\nabla e_i$ ( $k-m+1 \le j,i \le k$).  Terms of these forms are bounded using Lemma \ref{bbounds}, \eqref{wg1} and \eqref{eg1}:
\begin{align*}
\| w_{k+1} \cdot\nabla w_{k+1} \|_{V'} & \le M \| \nabla w_{k+1} \|^2 \le M \nu^{-2} \| g(u_k) \|_{V'}^2,\\
\| e_j\cdot\nabla w_{k+1}  \|_{V'} & \le M  \| \nabla e_j \| \| \nabla w_{k+1} \| \\
& \le   \frac{M \nu^{-2}}{1-\kappa} \| g(u_k) \|_{V'}  \left(  \| g(u_j) \|_{V'} +  \| g(u_{j-1}) \|_{V'} \right),
\\
\| w_{k+1}\cdot\nabla e_j  \|_{V'} & \le M  \| \nabla w_{k+1} \|  \| \nabla e_j \| \\ & \le   \frac{M \nu^{-2}}{1-\kappa} \| g(u_k) \|_{V'}  \left(  \| g(u_j) \|_{V'} +  \| g(u_{j-1}) \|_{V'} \right),
\\
\| e_j \cdot\nabla e_{i} \|_{V'} &  
\le M \| \nabla e_i \| \| \nabla e_j \| \\ 
& \le \frac{M\nu^{-2}}{(1-\kappa)^2}  \left(  \| g(u_i) \|_{V'} +  \| g(u_{i-1}) \|_{V'} \right)\left(  \| g(u_j) \|_{V'} +  \| g(u_{j-1}) \|_{V'} \right).
\end{align*}
Hence we can bound $R$ via
\begin{multline*}
\| R\|_{V'} \le 
\bar\alpha (1+\bar\alpha) M \nu^{-2} \| g(u_k) \|_{V'}^2
+ C(m) \bar\alpha^2 \frac{M \nu^{-2}}{1-\kappa} \| g(u_k) \|_{V'}  \sum_{j=k-m}^k  \| g(u_j) \|_{V'} \\
+ C(m) \bar\alpha^2 \frac{M \nu^{-2}}{(1-\kappa)^2}   \sum_{j=k-m}^k  \| g(u_j) \|_{V'}^2, 
\end{multline*}
where $C(m)$ is a constant depending quadratically on $m$.  Reducing completes the proof.  
\end{proof}

\subsection{More on the choice of the optimization norm}\label{sec:opt-norm}

The theory above relies on the choice of optimization norm being $V'$, which is due to the function $g(v)$ being in $V'$ and not $L^2$ (or $\ell^2$) due to the $v\cdot\nabla v$ term with $v\in V$. Note that the space the nonlinear residual $g(v)$ resides in is completely problem dependent, and different PDEs with different nonlinearity structures would have different `best' optimization norms and for other non-PDE problems the choice of $\ell^2$ may well be the best option.  

Our numerical experiments show that for NGMRES-Picard for NSE, using $V'$ optimization norm gives better convergence than $\ell^2$, especially in 3D, and we give now some explanation for this.  Note that for solving PDEs, $\ell^2$ is typically not the norm that matches the theory.  However, $\ell^2$ is closely related to $L^2(\Omega)$, which is a commonly used Hilbert space norm in PDE theory, and using the $\ell^2$ norm in place of the $L^2$ is reasonable in many situations (see \cite{hawkins2025choice} for some notable exceptions).  We discuss below how $L^2$ (and by extension $\ell^2$) is a reasonable choice for NGMRES-Picard in 2D, but not in 3D.

Consider the discrete setting, e.g. working in the weakly divergence-free space $V_h$ instead of $V$. The proof of Theorem \ref{thm0} does not change until \eqref{m0A1}.  Here, the problem term is $w^h_{k+1}\cdot\nabla w^{h}_{k+1}$ (with $h$ denoting discrete variables), we have from the inverse inequality \cite{BS08} that with the $L^2$ norm,
\begin{align*}
(2D): \| w_{k+1}^h \cdot\nabla w_{k+1}^h \| & \le C (\log(h^{-1}))^{1/2} \| \nabla w^h_{k+1} \|^2, \\
(3D): \| w_{k+1}^h \cdot\nabla w_{k+1}^h \| & \le C h^{-1/2} \| \nabla w^h_{k+1} \|^2.
\end{align*}
This leads to the estimates: in 2D, 
\begin{align*}
\| g(u_{k+1}^h) \| 
& \le \theta_{k+1}^{L^2} \| g(u_k^h) \|
+ C (\log(h^{-1})^{1/2} M \nu^{-2} | \bar \alpha| (|\bar \alpha|+1)\cdot \| g(u_k^h) \|_{V'}^2, \\
\| g(u_{k+1}^h) \|
& \le \gamma_{k+1}^{L^2}\kappa \| g(u_k^h) \| 
+ C (\log(h^{-1})^{1/2} M \nu^{-2} | \bar \alpha| (|\bar \alpha|+1)\cdot \| g(u_k^h) \|_{V'}^2.
\end{align*}
and in 3D,
\begin{align*}
\| g(u_{k+1}^h) \| 
& \le \theta_{k+1}^{L^2} \| g(u_k^h) \|
+ C h^{-1/2} M \nu^{-2} | \bar \alpha| (|\bar \alpha|+1)\cdot \| g(u_k^h) \|_{V'}^2, \\
\| g(u_{k+1}^h) \|
& \le \gamma_{k+1}^{L^2}\kappa \| g(u_k^h) \| 
+ C h^{-1/2} M \nu^{-2} | \bar \alpha| (|\bar \alpha|+1)\cdot \| g(u_k^h) \|_{V'}^2,
\end{align*}
noting that {\color{black} $\theta_{k+1}^{L^2}$} and $\gamma_{k+1}^{L^2}$ change since the $L^2$ norm is now assumed for the optimization problem.

While this discussion is for Theorem \ref{thm0}, the same terms (and more) arise in the proof of the Theorem \ref{thmm}, and so the discussion applies to the general $m$ results as well.  Hence we conclude that for 2D, using the $L^2$ norm is likely to be a good choice even for very small $h$ that will provide similar acceleration as when using $V'$.  However, in 3D, the $h^{-1/2}$ can be a large constant that makes the higher order terms large enough to dominate the linear terms and even prevent convergence.  Indeed, our numerical results show results using $\ell^2$ in 2D are similar to those using $V'$, but in 3D our results using $\ell^2$ are much worse in all cases than those found using $V'$.

\section{Numerical Experiments}
In this section we give results of several numerical tests that both illustrate our theory and test various properties of the method.  In all tests we use $(X_h,Q_h)=(P_r,P_{r-1}^{disc})$ Scott-Vogelius (SV) elements, with $r=d$. The meshes are constructed as barycenter refined triangle/tetrahedral meshes, and these elements are known to be inf-sup stable on such meshes \cite{arnold:qin:scott:vogelius:2D,GS19,JLMNR17,scott:vogelius:conforming, zhang:scott:vogelius:3D}.  All tests use an initial guess for velocity as zero in the interior but satisfying the boundary conditions.  To solve the linear systems that arise at each nonlinear iteration, we follow \cite{HR13,benzi,BL12} by using grad-div stabilization (which does not affect the solution since incompressibility is enforced strongly) and using an approximate block LU factorization as a preconditioner that approximates the Schur complement with the pressure mass matrix.  Inner solves are done with direct solvers.

This section will i) illustrate the convergence theory by calculating $\theta_k$ for several computations and comparing it to the ratio of successive $\| g(u_k)\|_{V'}$, ii) show that changing $m$ from small at early iterations to large at later iterations can be a good strategy, iii) show that using the $V'$ norm in the numerical tests to match the theory gives better results than using $\ell^2$, and iv) show that the convergence properties are mesh independent.  This will be done by running tests on multiple test problems with varying $m$, $Re$ and $h$.  For convergence, we plot the nonlinear residual $\| g(u_k)\|_{V'}$ vs iteration $k$, since this is what the theory gives results for. We also plotted the Picard fixed point residual $\| \nabla ( q(u_k) - u_k)\|$ vs iteration $k$, and found it always had the same convergence pattern as $\| g(u_k)\|_{V'}$ but values were larger (proportional to $Re$). %but was approximately two orders of magnitude (or so) larger.  
This relationship between $\| g(u_k)\|_{V'}$ and $\| \nabla ( q(u_k) - u_k)\|$ is expected due to Lemma {\color{black} 2.1}, and plots of convergence for $\| \nabla ( q(u_k) - u_k)\|$ are not given in the text. {\color{black} Additionally, convergence plots for $\|g(u_k)\|_{\ell^2}$ were also calculated, and had the exact same behavior as for $\| g(u_k)\|_{V'}$ except shifted up by roughly an order of magnitude (plots omitted).}

\subsection{Test Problem 1: 2D driven cavity}

\begin{figure}[ht]
\center
 $Re=5000$ \hspace{1in}  $Re=10000$ \\
\includegraphics[width = .3\textwidth, height=.28\textwidth,viewport=115 45 465 390, clip]{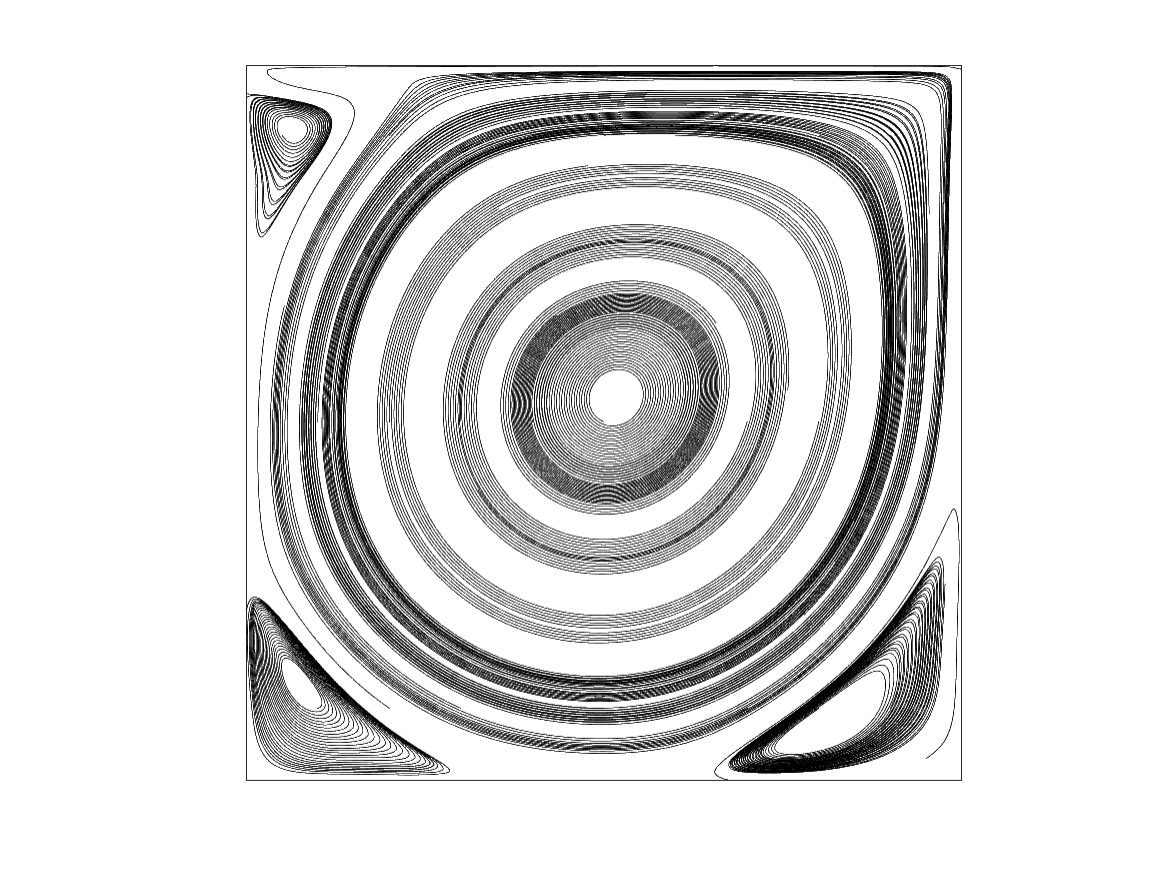}  
\includegraphics[width = .3\textwidth, height=.28\textwidth,viewport=115 45 465 390, clip]{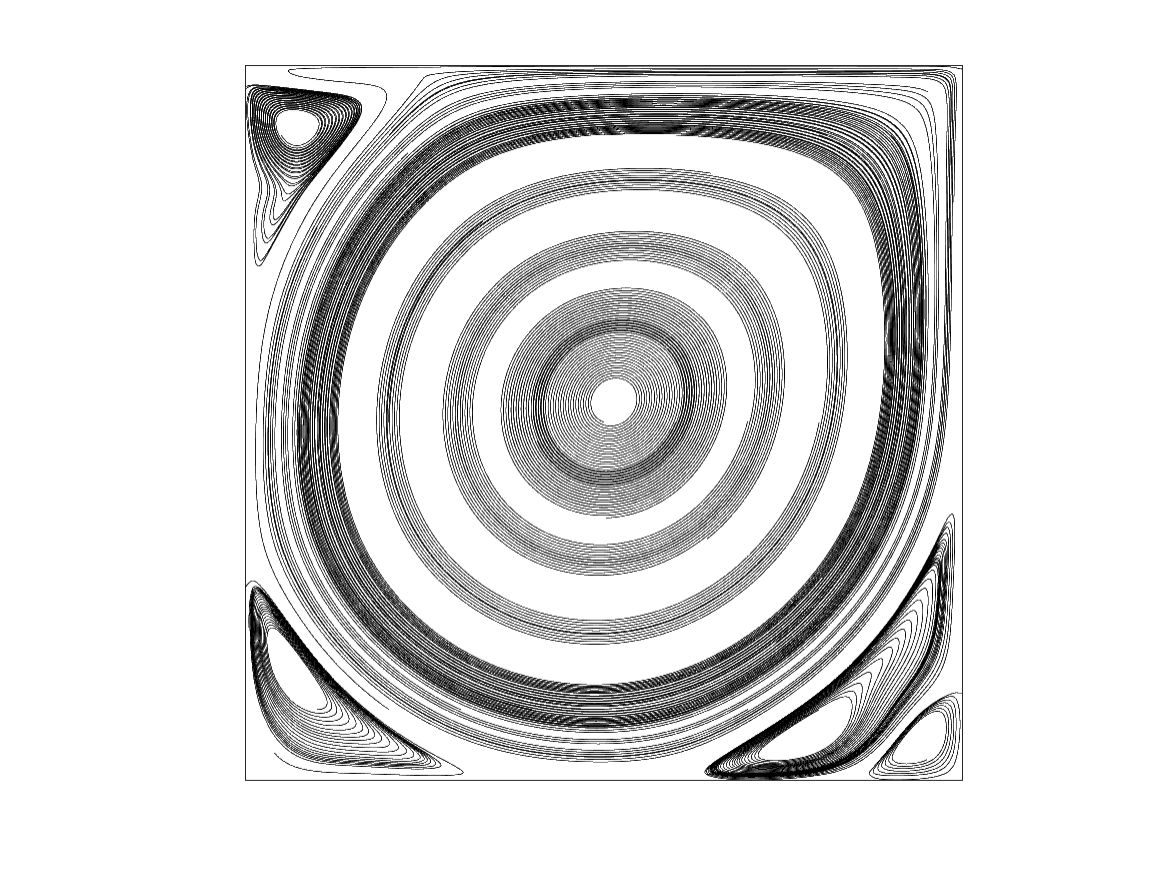}  
\caption{\label{dc2d} The plots above shows 2D driven cavity solutions as velocity streamlines, for varying $Re$.}
\end{figure}

The 2D driven cavity is a classical benchmark test problem that has domain $\Omega=(0,1)^2$,  no forcing $f=0$, homogeneous Dirichlet velocity boundary conditions on the sides and bottom and a moving lid $[ 1,0 ]^T$ on the top.  For this problem, $Re=\nu^{-1}$.  We compute with $Re$=5000 and 10000, and we note that our solutions were found to be in agreement with those from the literature \cite{ECG05}, see Figure \ref{dc2d}.   For discretizations, we use $h=\frac{1}{64},\ \frac{1}{128},\ \frac{1}{196},\ \frac{1}{392}$ meshes built from barycenter refinements of uniform meshes; the finest mesh has 923K velocity degrees of freedom (dof) when equipped with $(P_2,P_1^{disc})$ SV elements.  The stopping tolerance was the nonlinear residual falling below $10^{-8}$ in the $V'$ norm.  We test several aspects of NGMRES-Picard on this test problem.

\subsubsection{Illustration of convergence theory}

\begin{figure}[h]
\center
\includegraphics[width = .48\textwidth, height=.3\textwidth,viewport=0 0 750 410, clip]{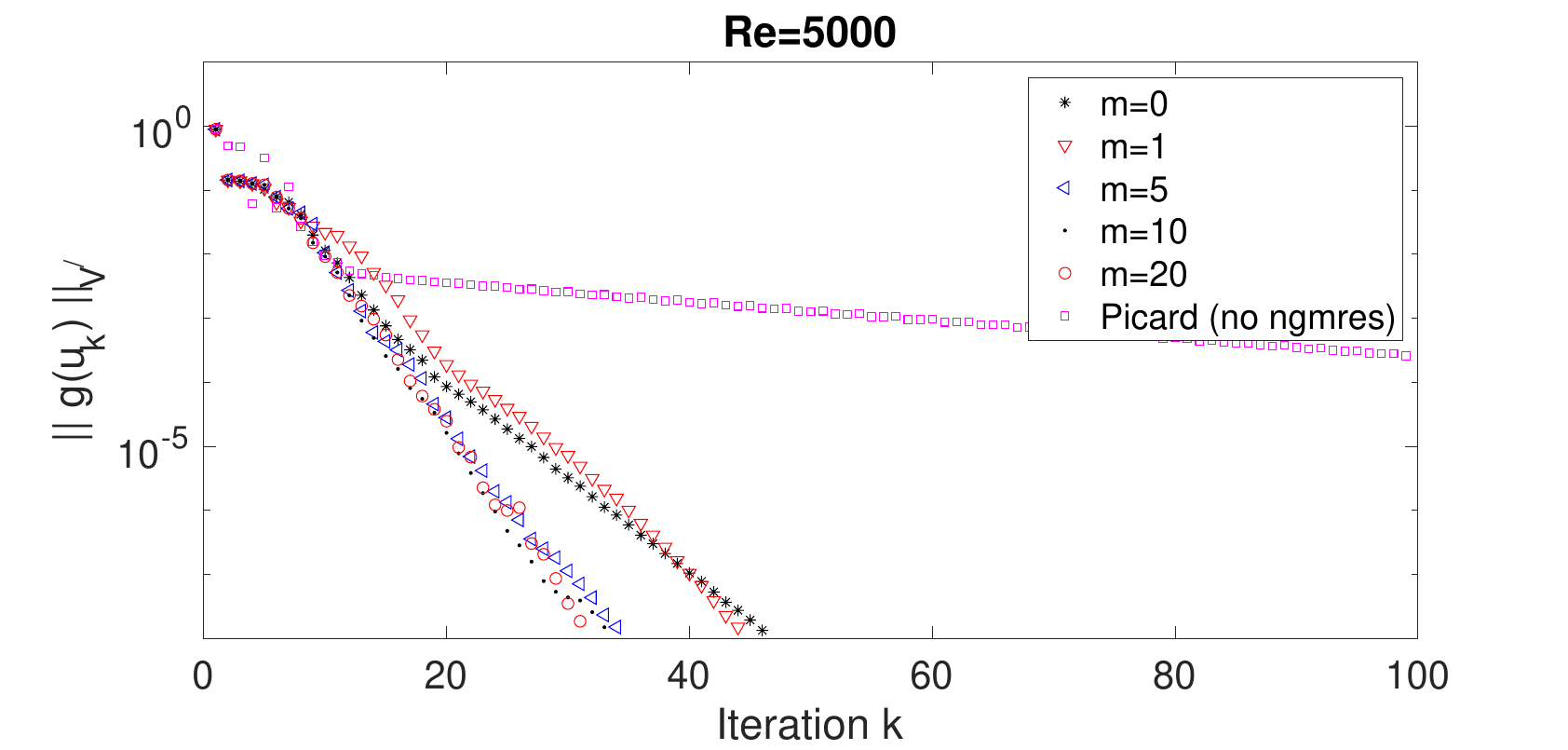}  
\includegraphics[width = .48\textwidth, height=.3\textwidth,viewport=0 0 750 410, clip]{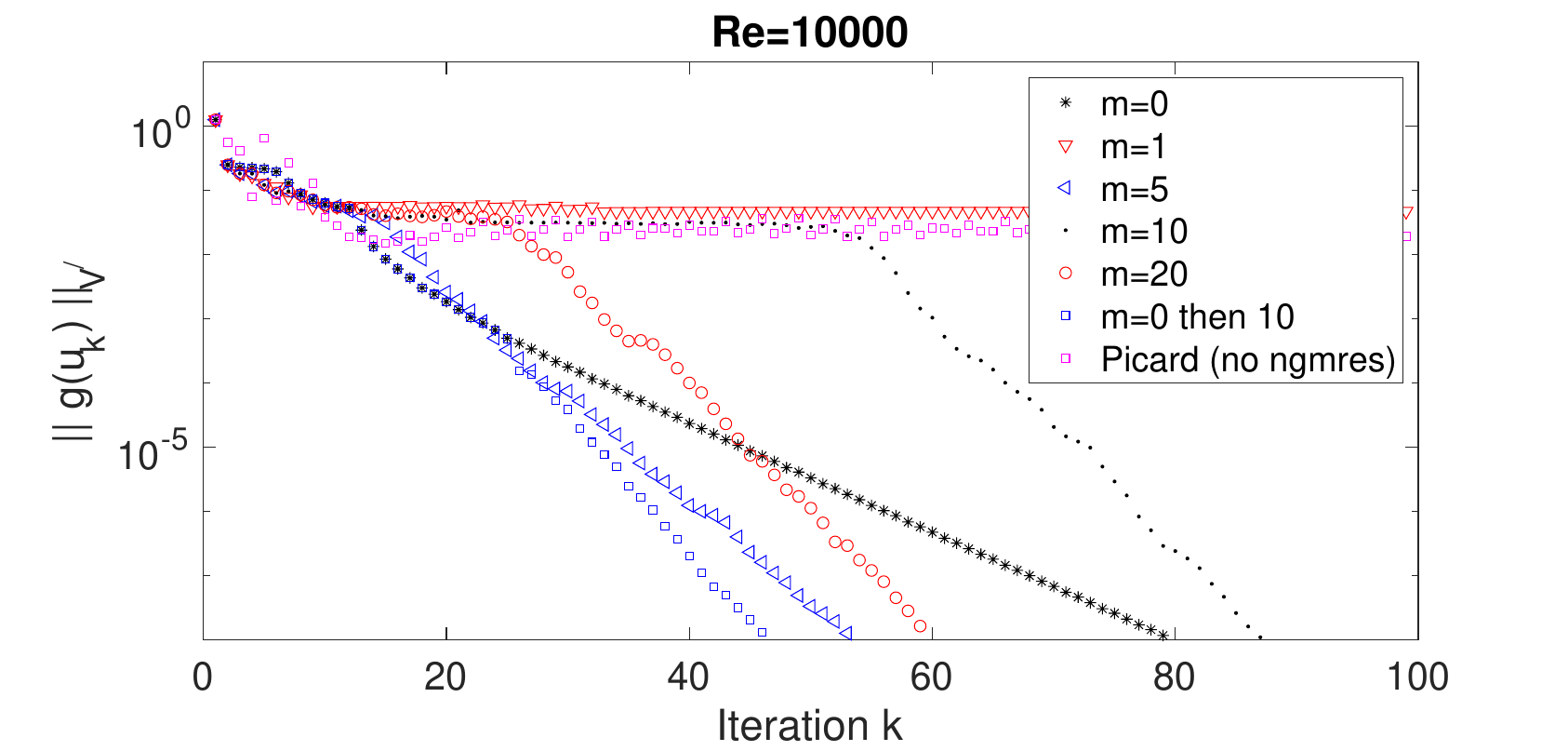}  \\
\includegraphics[width = .48\textwidth, height=.3\textwidth,viewport=0 0 750 410, clip]{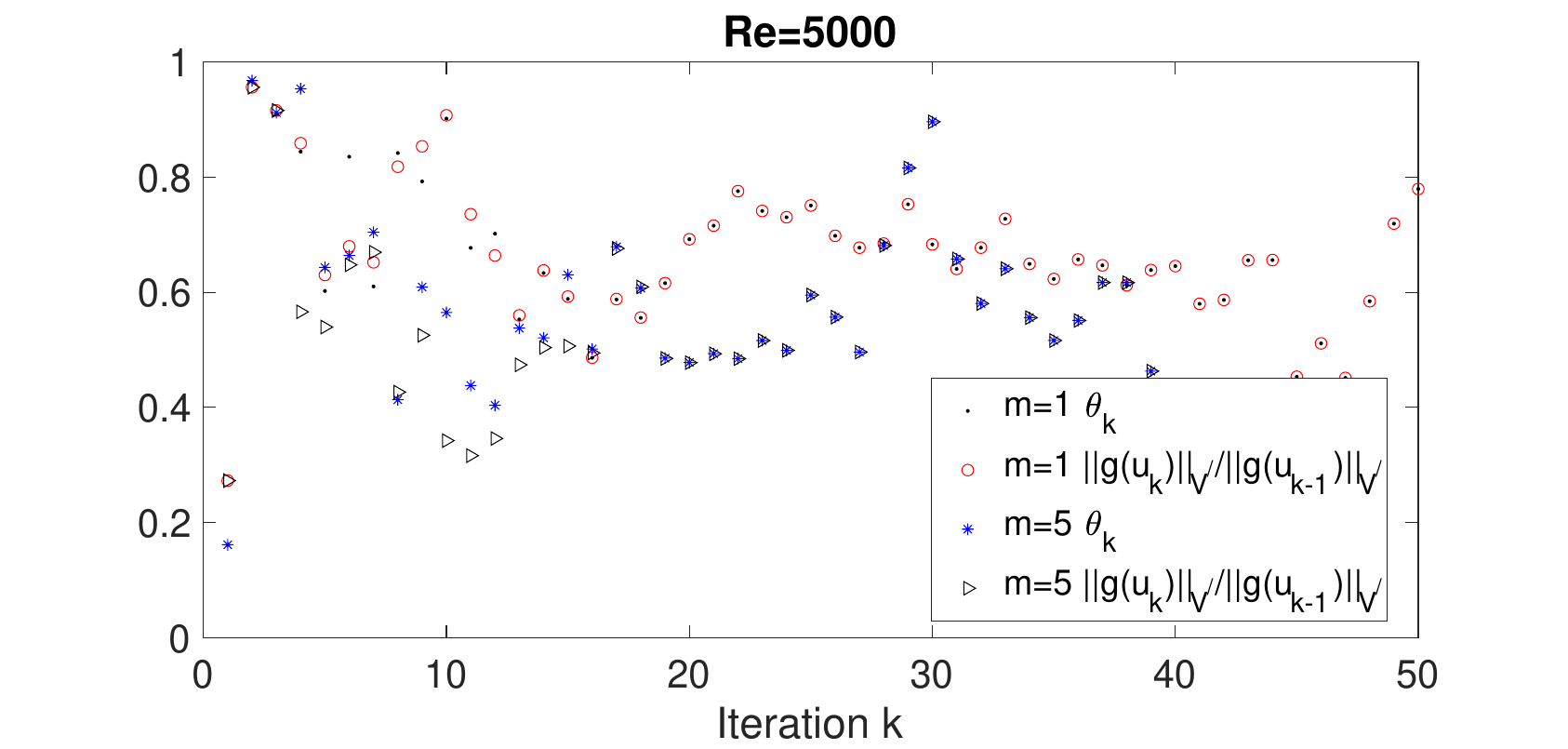}  
\includegraphics[width = .48\textwidth, height=.3\textwidth,viewport=0 0 750 410, clip]{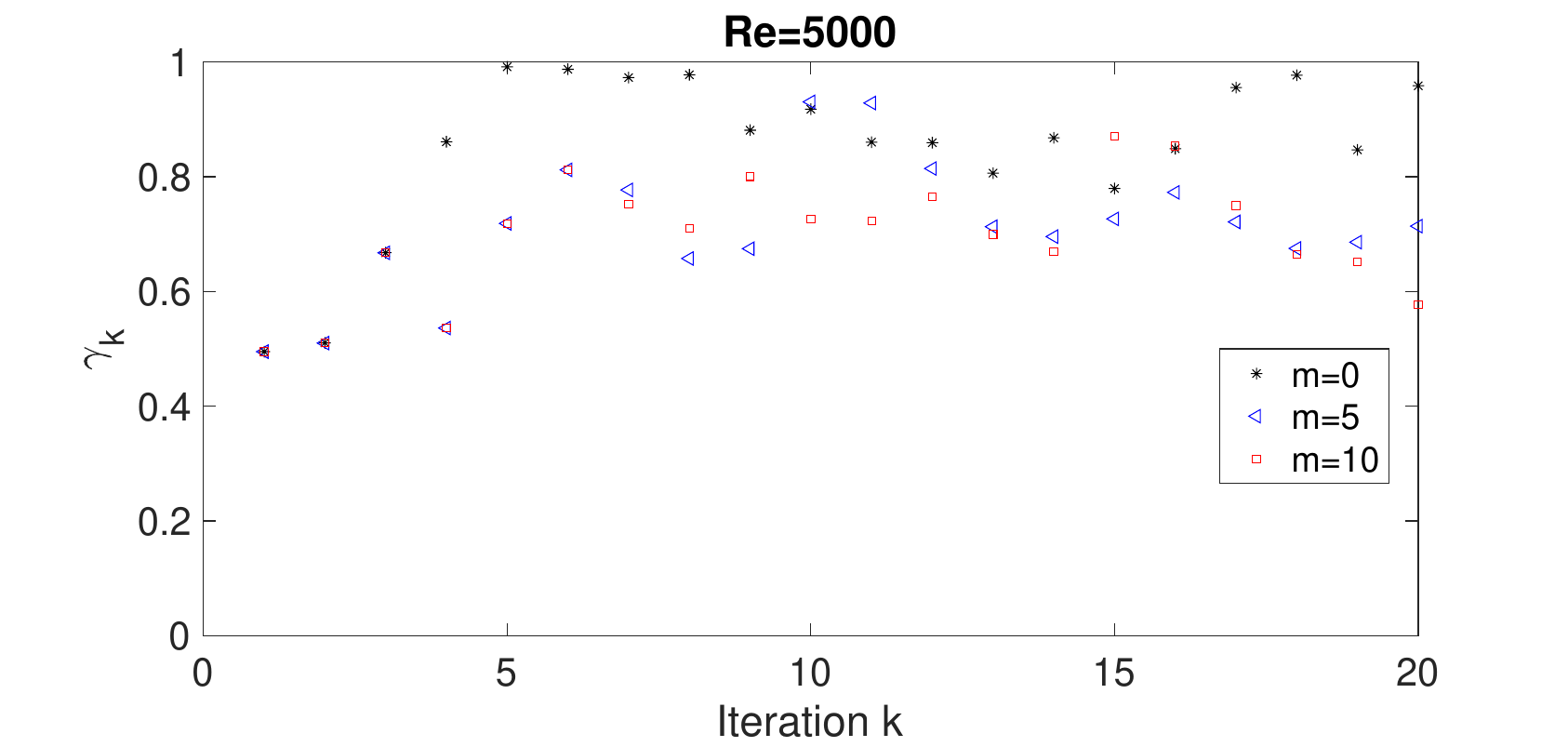}  
\caption{\label{dcconv} The plots above show (top) convergence of NGMRES for NSE 2D driven cavity for varying $m$, and (bottom) plots of $\theta_k$ and $\gamma_k$  compared to the ratio $\| g(u_k) \|_{V'} / \| g(u_{k-1}) \|_{V'}$.}
\end{figure}

For $Re$=5000 and 10000, we test convergence for varying $m$ using $h=\frac{1}{256}$, and results are shown in Figure \ref{dcconv}.  Results were run on the other meshes, and similar results were found; this is discussed in more detail in Section {\color{black}4}.1.2.

For $Re$=5000, we observed that while usual Picard does not converge in 100 iterations, it has a steady linear decline and we suspect it is contractive in this case and will eventually converge (perhaps in another 600 iterations or so).  Convergence is dramatically improved by NGMRES, as we observe NGMRES-Picard with $m=0$ and $m=1$ converge in about 45 iterations and $m=5,10,20$ converge in 30 iterations or so.  

For $Re$=10000, we observe that usual Picard does not reduce the residual in the first 100 iterations, and from the plot we expect it will never converge (but will remain bounded, as is discussed in Section 2).  NGMRES-Picard, however, does converge (except for $m=1$), and convergence with $m=0$ and $m=10$ takes about the same number of iterations (although their paths to convergence were quite different), and $m=5$ and $m=20$ showed similar convergence behavior.  Our convergence theory shows there are higher order terms which can be large/dominant when the residual is large, and the higher $m$ is, the more higher order terms there are; this helps {\color{black}explain} why the slope for $m=0$ is best at the early iterations but after the residual gets smaller the convergence is faster for larger $m$.  When $m=1$ there is no convergence, perhaps because the early iterations were not good and pushed it into a direction from which it could not recover.  We also ran a test where $m=0$ was used until the residual fell below $10^{-3}$ and then switched to $m=10$; this gave the best convergence results.

The theory above from Theorem \ref{thmm} proves that
\begin{align*}
\| g(u_{k}) \|_{V'} & \le \theta_{k} \| g(u_{k-1})   \|_{V'} + \mbox{ higher order terms, and} \\
\| g(u_{k}) \|_{V'} & \le \gamma_{k} \kappa \| g(u_{k-1})   \|_{V'} + \mbox{ higher order terms}.
\end{align*}
To test the sharpness of $\theta_k$, we compare $\theta_k$ to $\frac{\| g(u_{k}) \|_{V'}}{\| g(u_{k-1})   \|_{V'}}$ in Figure \ref{dcconv} (bottom left) for varying $m$ and $Re$=5000.  We observe some variation in these quantities at early iterations, but for larger $k$ (i.e. once the residual gets small) we see agreement that is visually indistinguishable in the plots.  This shows the theory is quite sharp for $\theta_k$, once the iteration gets close to the root (we tested this using different $m$ and $Re$=10000, and the same sharp agreement was found).  We also observe, as expected, that $\theta_k$ for $m=5$ are overall smaller than for $m=1$.  Our theory is for 1 step, so it does not apply directly to make sure each $\theta_k$ is smaller at each step, but overall the minimization using $m=5$ is over a bigger space and thus a better gain of the optimization problem is expected.

We also plot $\gamma_k$ in Figure \ref{dcconv} at bottom right for varying $m$ for the first 20 iterations (to make viewing the plot easier).  $\gamma_k$ drives the acceleration that NGMRES provides, and from the plot we observe that as $m$ gets larger, $\gamma_k$ gets smaller.  Since we never know $\kappa$ exactly, knowing the $\gamma_k$ is not as sharp of a predictor as $\theta_k$ for residual reduction.

\subsubsection{Test of mesh independence}

\begin{figure}[h]
\center  
\includegraphics[width = .55\textwidth, height=.35\textwidth,viewport=0 0 750 410, clip]{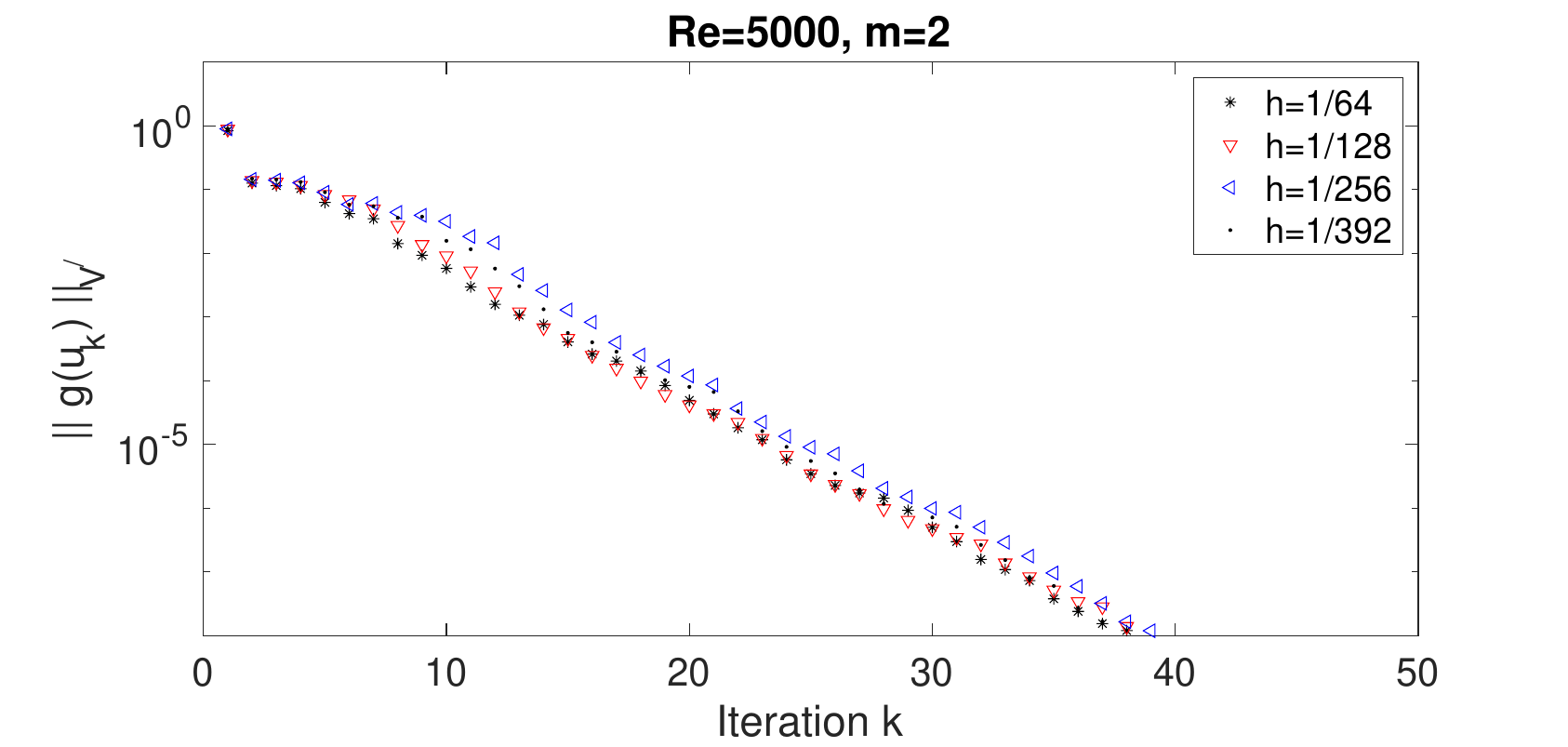}  
\caption{\label{meshind} The plot above show convergence results for $Re$=5000 with varying mesh widths $h$, and no mesh dependence is observed.}
\end{figure}

We now numerically test the mesh independence of the convergence, and note that our theory showed no mesh dependence for the convergence.  To test this, we chose $Re$=5000 and $m$=2, and ran tests for four successively refined meshes.  Results are shown in Figure \ref{meshind}, and for each $Re$ we observe almost exactly the same convergence behavior on the varying meshes. We reran this test with varying $Re$ and $m$ and always saw the same mesh independent convergence.

\subsubsection{Test of $V'$ optimization norm versus $\ell^2$}

\begin{figure}[h!]
\center
\includegraphics[width = .48\textwidth, height=.3\textwidth,viewport=0 0 750 410, clip]{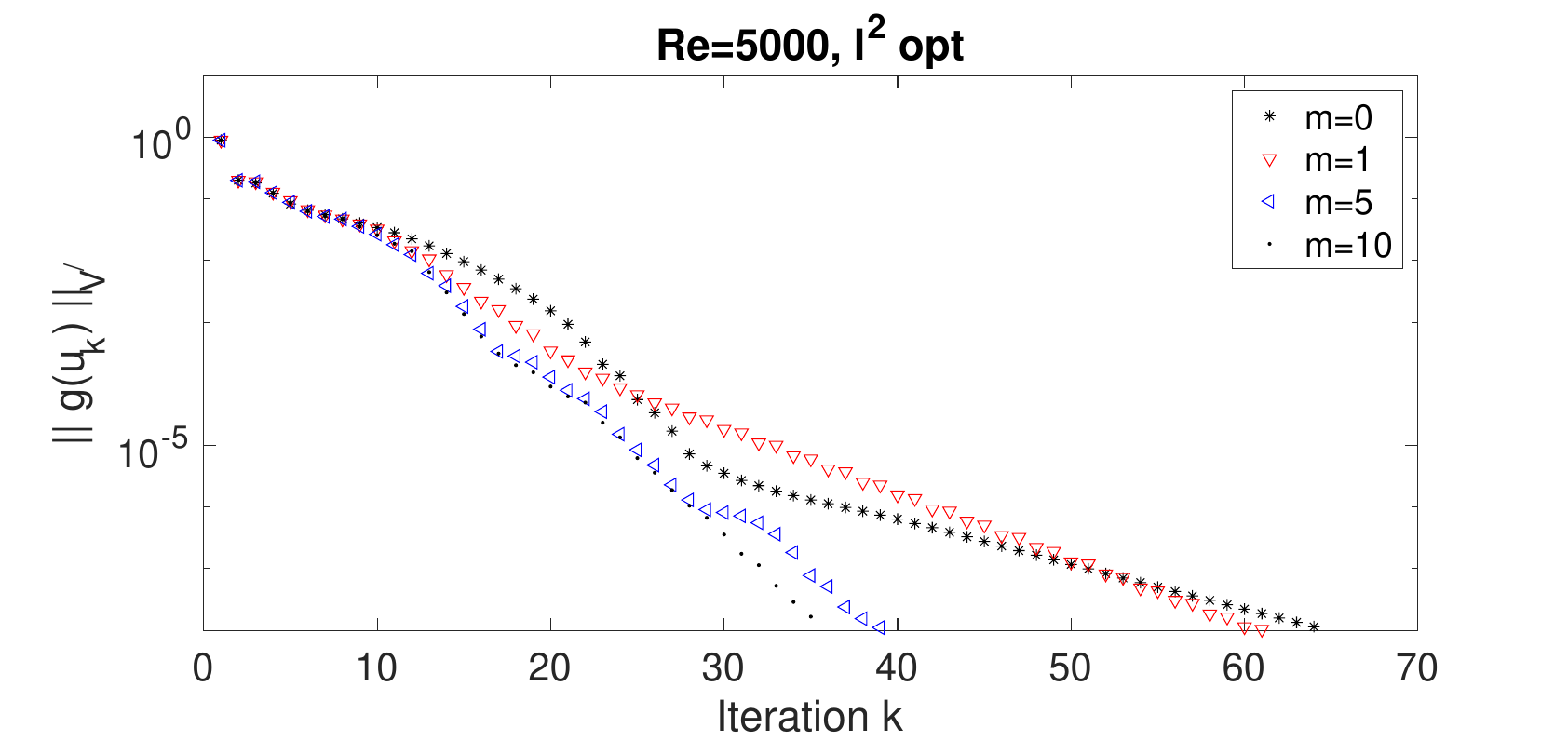}  
\includegraphics[width = .48\textwidth, height=.3\textwidth,viewport=0 0 750 410, clip]{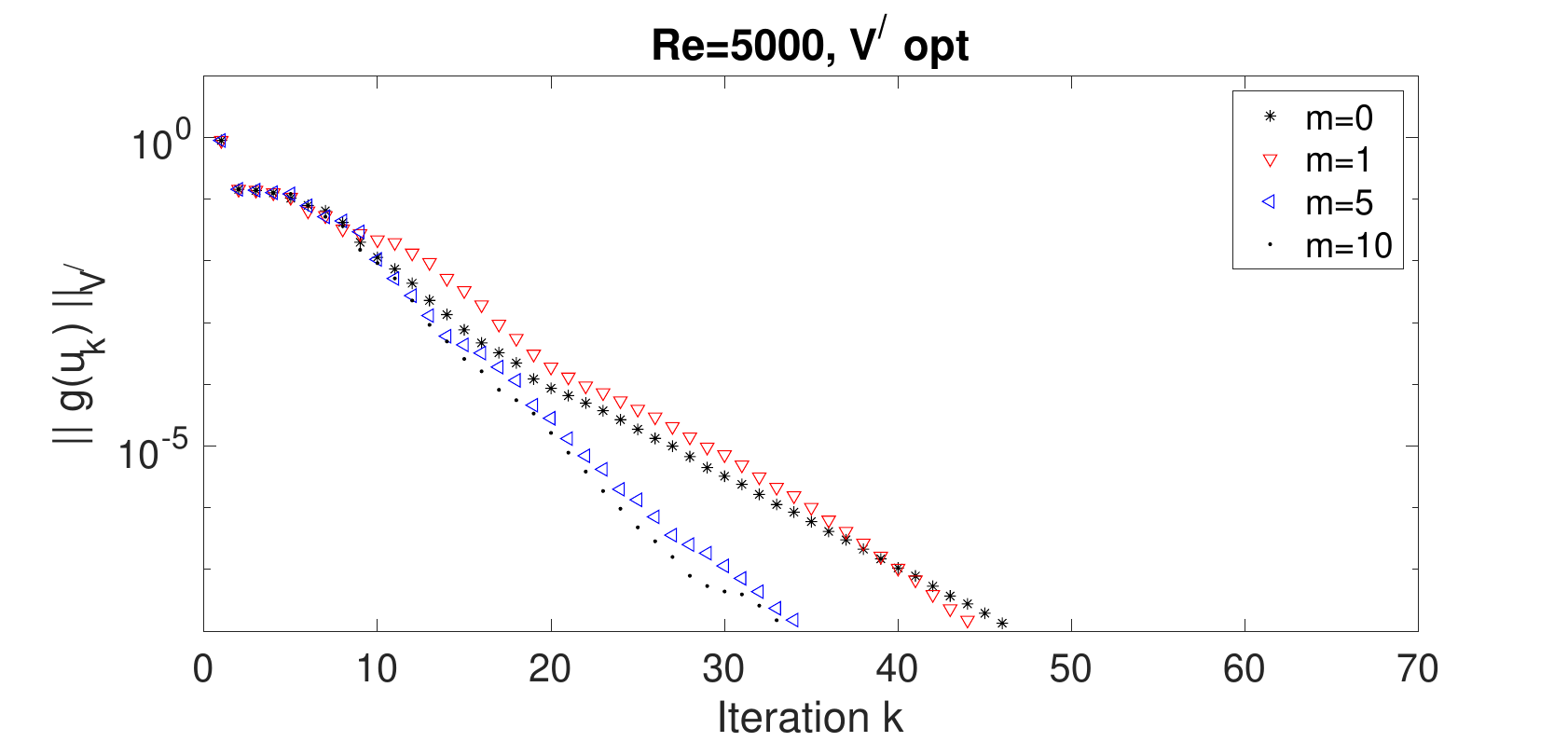}  
\caption{\label{dcconvopt} The plots above show convergence of NGMRES-Picard for NSE 2D driven cavity for varying $m$, using $Re$=5000, and $\ell^2$ optimization norm (left) and $V'$ optimization norm (right).}
\end{figure}

The convergence theory herein relies on the optimization norm being $\| \cdot \|_{V'}$. However, by far the most common choice for the NGMRES optimization problem in computations is $\ell^2$, due to its ease of use and long history of success.  Indeed, this paper is the first complete (single-step) convergence theory for NGMRES, and is also the first paper to suggest that the optimization norm should be anything other than $\ell^2$ (note that using $V'$ herein is specific to this problem since this is the space where the nonlinear residual resides, and so in other problems we expect the optimization norm to change accordingly).  

{\color{black} Using $V'$ requires an extra Stokes solve at each iteration, and while this is significantly cheaper than a Picard solve it is not negligible like $\ell^2$ would be.  More specifically, in these tests, a Picard solve takes on average 5.70 seconds, and a Stokes solve takes on average 0.93 seconds.  Details of the Picard solve are given above, and for Stokes we use a block LU factorization (e.g. as in \cite{benzi}) and need to perform 2 solves with the 1,1 block and 1 Schur complement solve.  Since it is the same Stokes matrix in each iteration, we precompute the Stokes matrix and pre-factor the 1,1 block.  The Schur complement is solved using conjugate gradient, using the augmented Lagrangian preconditioning approach of \cite{benzi} with tolerance $10^{-8}$.  On finer meshes, we would expect the relative cost of a Stokes solve to decrease due to the precomputed matrix assembly and 1,1 block factorization.
}

Hence we test NGMRES-Picard for the NSE using $V'$ and $\ell^2$ optimization norms for $Re$=5000, varying $m$, and $h=\frac{1}{256}$.  Results are shown in Figure \ref{dcconvopt}.  We observe that for $m=0$ and $m=1$, convergence is significantly better when the $V'$ norm is used.  However, for $m=5$ and $m=10$, the results show no significant difference when the different optimization norms are used.  This is consistent with our discussion in Section 3.3 about results from the different optimization norm choices.

\subsection{3D driven cavity}

\begin{figure}[h]
\center
$Re$=1000 \\
\includegraphics[width = .9\textwidth, height=.27\textwidth,viewport=120 0 1180 340, clip]{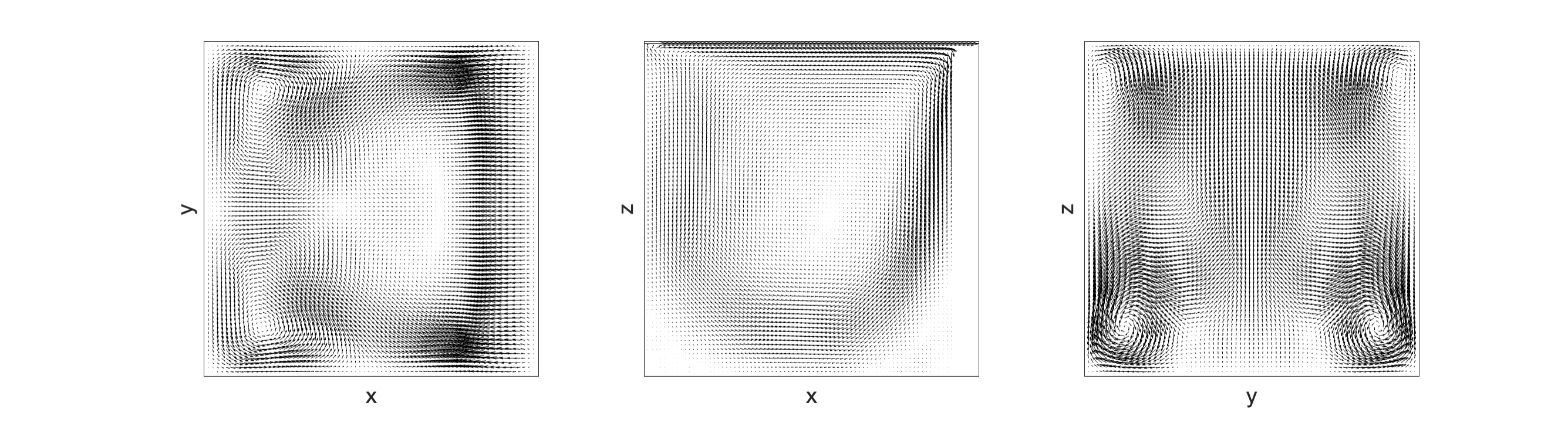}  \\
\caption{\label{cav3d} The plot above shows the $Re$=1000 3D driven cavity velocity solution displayed as midsliceplanes of velocity.}
\end{figure}

For our next test problem, we use the 3D driven cavity benchmark test.  This is an analogue to the 2D problem above, where again there is no external forcing, the domain is the unit cube, and 
homogeneous Dirichlet boundary conditions are enforced for velocity on the walls but $[1, 0, 0]^T$ is enforced as the `lid'.  The mesh is constructed by starting with Chebychev points on $[0, 1]$ to create a $ \mathcal{M}^3$ grid of rectangular boxes. Each box is split into 6 tetrahedra, and then each tetrahedra is barycenter refined into 4 tetrahedra. Computations used $\mathcal{M}=$11 and 13, which provided 796K and 1.3M dof respectively; as above, we found no significant difference in results on the different meshes and show only results for the finer mesh.  $Re=\nu^{-1}$ for this test, and we tested using $Re$=400 and 1000. Midspliceplane velocity vector solutions for $Re = 1000$ are shown in Figure \ref{cav3d}, and these match those from the literature well \cite{WB02}.  The linear solver tolerances of the outer solvers were set as $10^{-10}$, and a convergence tolerance of $10^{-7}$ in the $V'$ norm is used.

\subsubsection{Illustration of convergence theory}

\begin{figure}[h!]
\center
\includegraphics[width = .48\textwidth, height=.32\textwidth,viewport=0 0 750 410, clip]{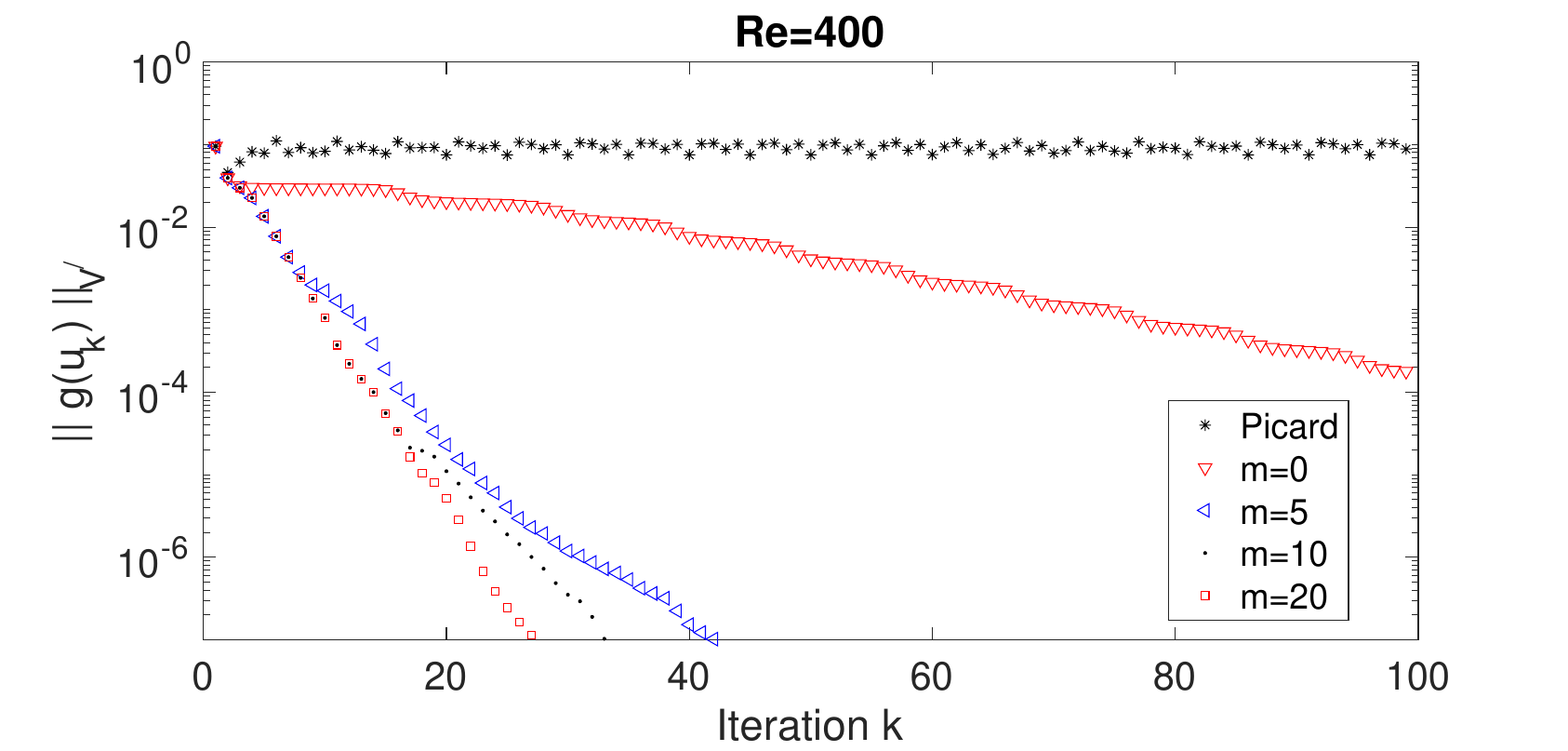}  
\includegraphics[width = .48\textwidth, height=.32\textwidth,viewport=50 0 750 410, clip]{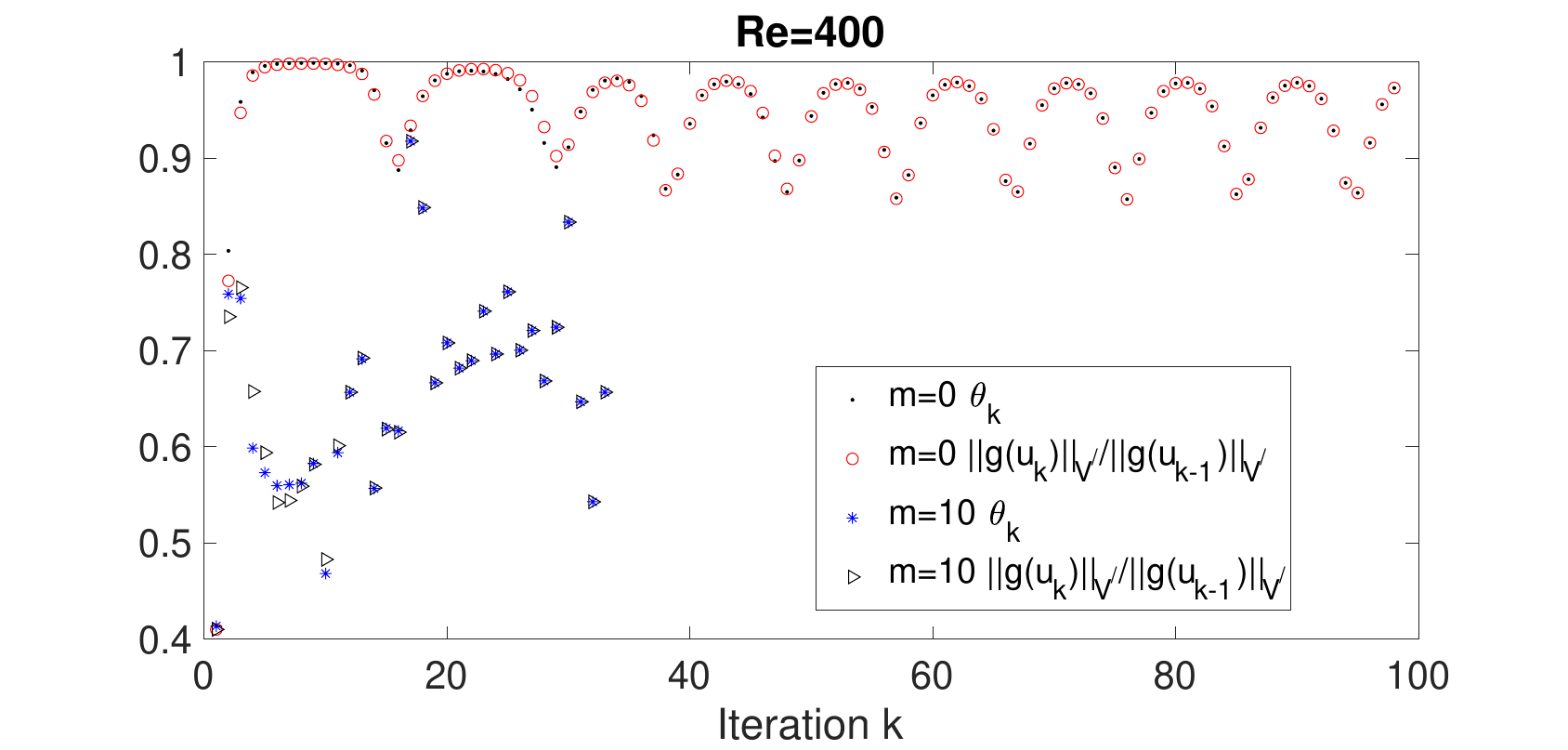}  \\
\includegraphics[width = .48\textwidth, height=.32\textwidth,viewport=0 0 750 410, clip]{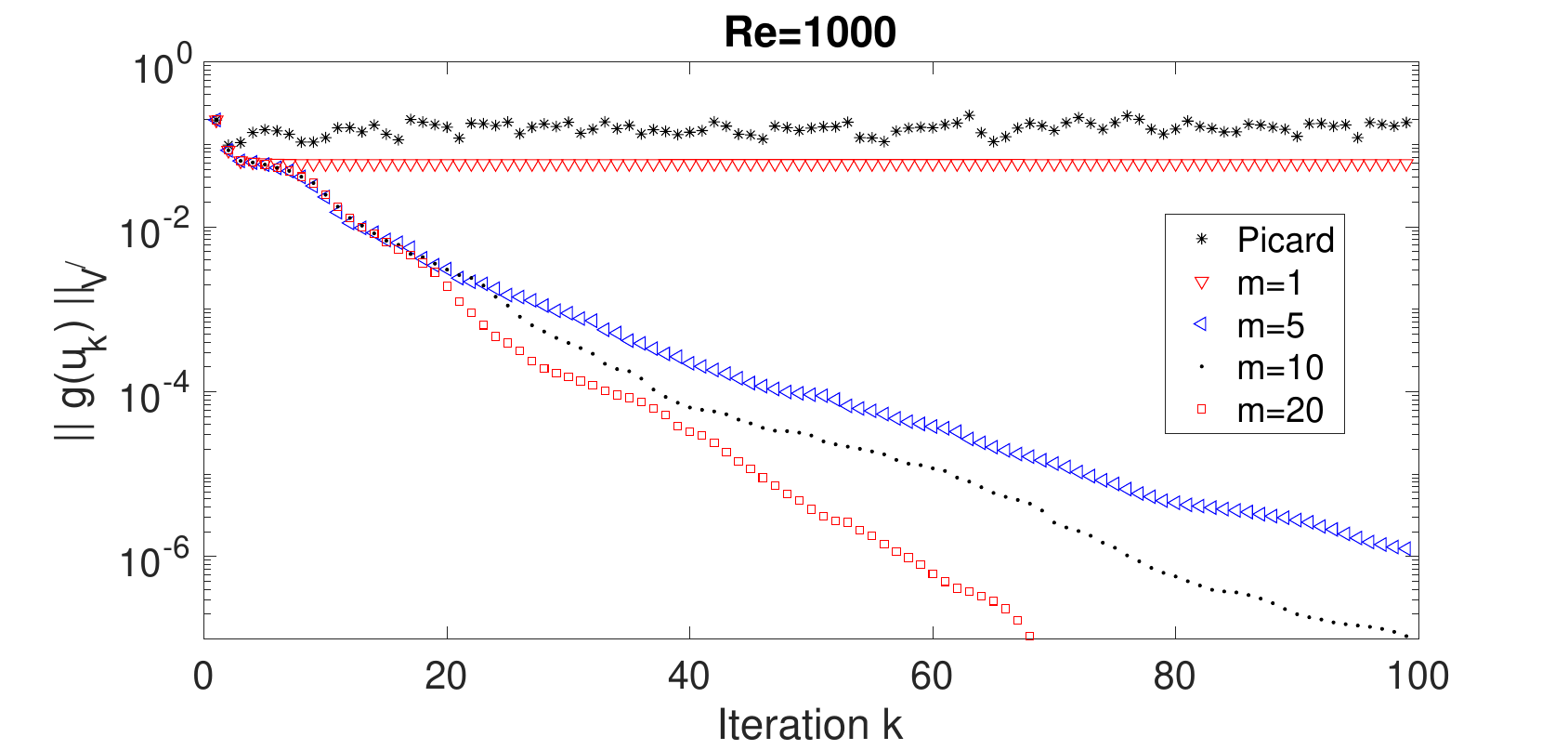}  
\includegraphics[width = .48\textwidth, height=.32\textwidth,viewport=50 0 750 410, clip]{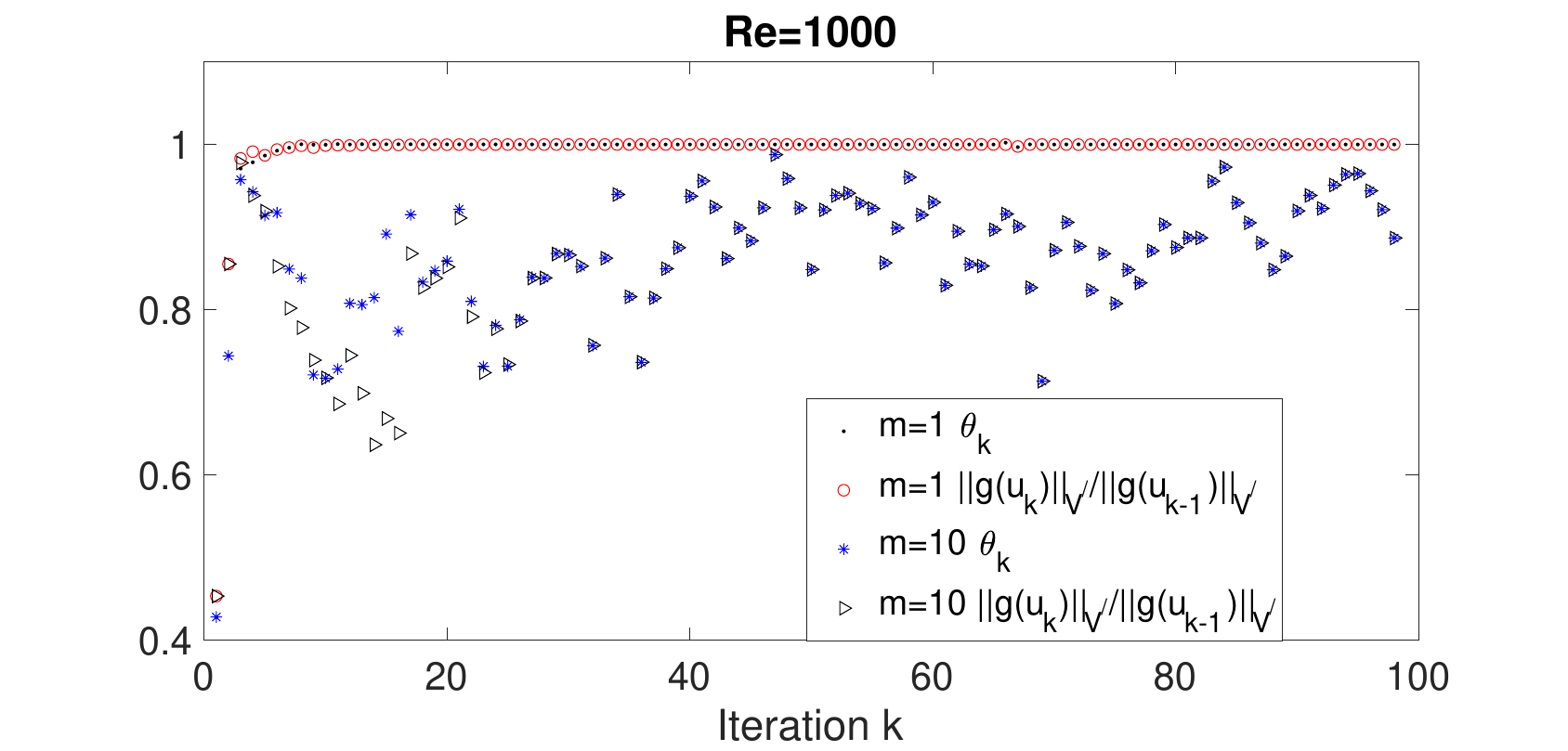}  
\caption{\label{dc3conv} The plots above show (left) convergence of NGMRES-Picard for NSE 3D driven cavity, for varying $m$ and (right) $\theta_k$ and the ratio $\| g(u_k) \|_{V'} / \| g(u_{k-1}) \|_{V'}$, for $Re$=400 and 1000.}
\end{figure}

Convergence results for the 3D driven cavity tests are shown in Figure \ref{dc3conv}, with the $V'$ optimization norm being used.  As expected from \cite{PRTX25}, Picard fails to converge in 100 iterations for both $Re$=400 and 1000.  Moreover, we note the converge plot is essentially flat and it does not appear that these iterations will ever converge.  For $Re$=400, adding NGMRES provides improvement in convergence even with $m=0$ (although it does not converge in 100 iterations, it has a clear downward trend), and each of $m=5,\ 10,\ 20$ all enable rapid convergence (with modest improvement for $m=20$ over $m=10$, and modest improvement of $m=10$ over $m=5$).  For $Re$=1000, $m=1$ fails to convergence or even show a negative slope in the convergence plot.  Convergence improvement by NGMRES is observed for $m\ge 5$: $m=10$ is able to convergence in 98 iterations, and $m=20$ converges in 68 iterations.  Hence for both $Re$=400 and 1000, taking $m=20$ provides the best results.

Also shown in Figure \ref{dc3conv} at right are plots of $\theta_k$ and the ratio $\frac{\| g(u_{k}) \|_{V'}}{\| g(u_{k-1})   \|_{V'}}$.  We observe, in agreement with the theory and just as in the 2D driven cavity problem, that after some initial iterations these quantities agree to a high precision and are visually indistinguishable in the plots.

\subsubsection{Test of $V'$ optimization norm versus $\ell^2$}
\begin{figure}[h!]
\center
\includegraphics[width = .48\textwidth, height=.32\textwidth,viewport=0 0 750 410, clip]{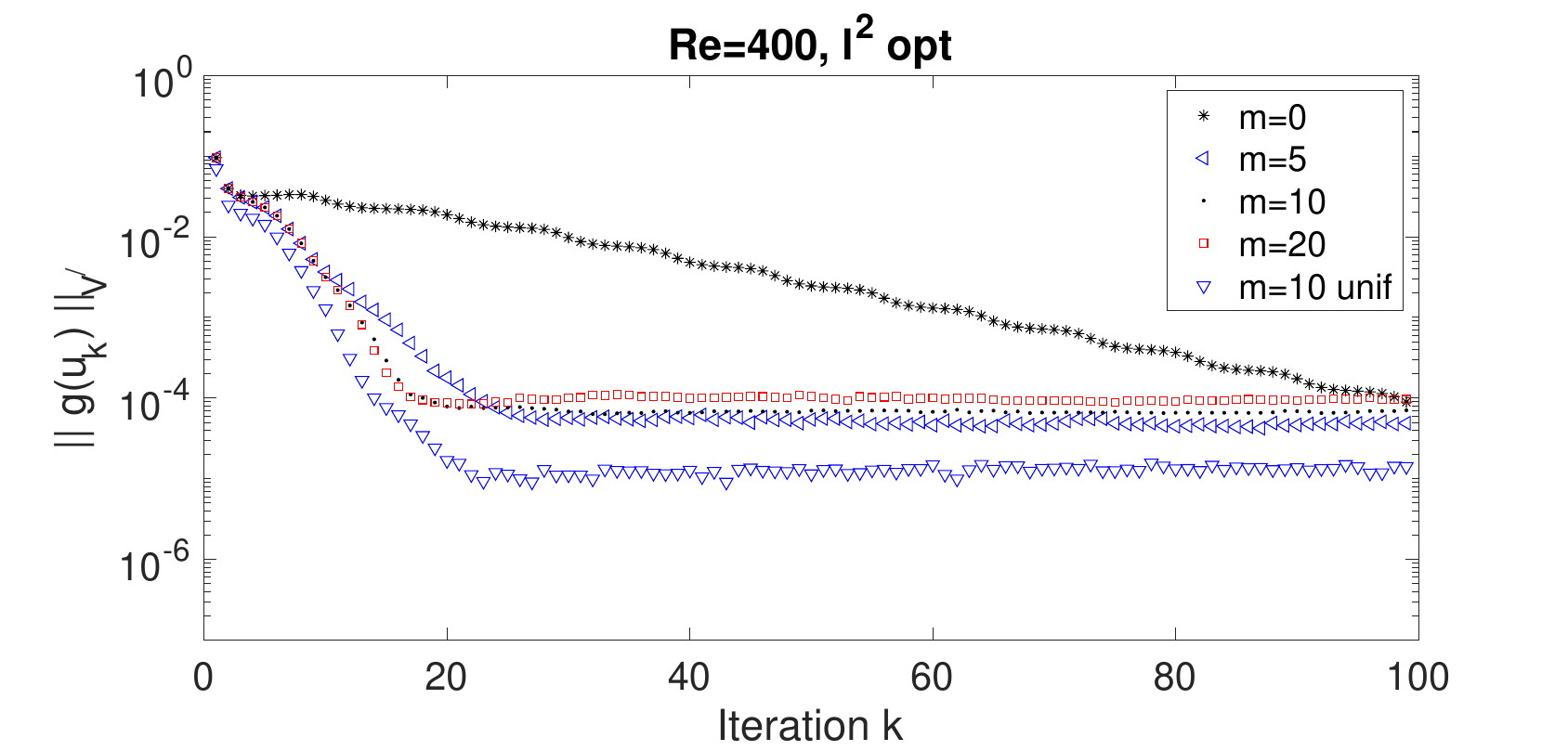}  
\includegraphics[width = .48\textwidth, height=.32\textwidth,viewport=0 0 750 410, clip]{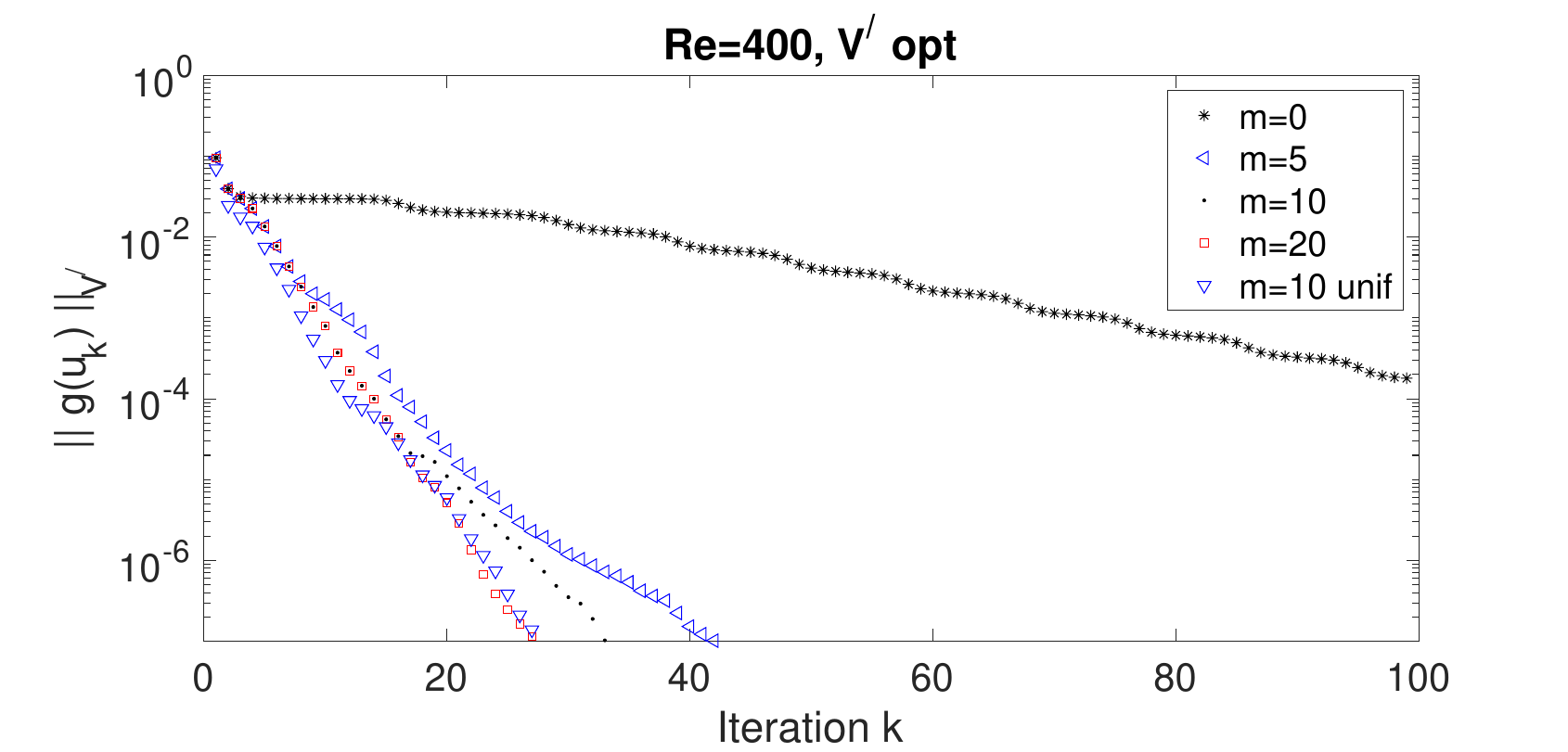}  
\caption{\label{dc3convopt} The plots above show convergence of NGMRES-Picard for NSE 3D driven cavity for varying $m$, using $Re$=400, and $\ell^2$ optimization norm (left) and $V'$ optimization norm (right).  The plots labeled `unif' use a uniform mesh instead of one generated from Chebychev points.}
\end{figure}

We now test convergence with the different optimization norms using $Re$=400.  Results are shown in Figure \ref{dc3convopt}, and we observe that convergence results using $\ell^2$ as the optimization norm are poor in comparison to using $V'$.  For each $m$ using $\ell^2$, the error reaches a minimum of around $10^{-4}$ (before it reaches this minimum, convergence behavior is similar to that found using $V'$ optimization norm).  This same failure to achieve convergence and reaching a minimum when the $\ell^2$ norm is used for optimization was found using $Re$=1000 and on different meshes.  While we have no theory for the 3D case using this norm,  in Section \ref{sec:opt-norm} we show how in 3D discretizations the convergence theory of this paper can fall apart if the $L^2$ optimization norm is used since the nonlinear term can no longer be bounded without the inverse inequality (which blows up the estimate).  

We also ran one additional test for both optimization norm choices, where we used a uniform spacing of the initial rectangular box discretization instead a Chebychev spacing.  We used $m=10$ for this test, and in the plots it is labeled as `$m=10$ unif'.   For the case of 
$V'$ optimization norm, this change had little effect on convergence, as expected.   However, for the $\ell^2$ optimization norm, the error now bottoms out near $10^{-5}$, suggesting the uniform mesh and larger minimum mesh width helped the convergence.  This is consistent with the discussion in Section \ref{sec:opt-norm} about using the $\ell^2$ optimization norm in 3D for NGMRES-Picard.

{\color{black}
\subsubsection{Comparison to AA}
As discussed in the introduction, AA is similar to NGMRES in that it is an extrapolation method based on a residual-minimizing optimization problem.  AA takes the form, for depth $\hat m\le k+1$,
\[
 				u_{k+1} = \sum_{i=k-\hat m}^{k+1} \hat \alpha^{k+1}_i q(u_{i-1}),
\]
where $\big(\hat \alpha^{k+1}_{k+1},\ \hat\alpha^{k+1}_{k},\cdots,\ \hat\alpha^{k+1}_{k-\hat m}\big)$ solve the least-squares optimization problem
 \begin{equation}\label{eq:min-NG3} 
 				\min_{\sum_{i=k-\hat m}^{k+1} \hat \alpha_i^{k+1} = 1} 
				\|  \sum_{i=k-\hat m}^{k+1} \hat\alpha^{k+1}_i (q(u_{i-1})-u_{i-1}) \|_{X}^2.
 \end{equation}
Here, $X$ is the appropriate norm since it is the range norm of $q$ \cite{PRX19,HR25}.  The key differences between AA and NMGRES are that the extrapolation is different and that AA uses a fixed-point residual while NGMRES uses the true nonlinear residual.  These differences lead to significant differences in the theory above compared to AA theory in \cite{PRX19,PR25}, including that NGMRES' convergence bound has better higher order terms than AA (e.g. for depth 0, NGMRES higher order terms are quadratic while for AA they are less than quadratic) and the NGMRES optimization problem implicitly creates a term that accurately predicts the linear convergence rate.

It is an open problem to classify what types of problems/iterations are significantly better suited to be used with NGMRES over AA (or vice versa), and so a comparison of the two methods for Picard for NSE is in order.  Hence we now test the 3D driven cavity with $Re$=400 and 1000 using both NGMRES and AA, for varying depths.  Note that depth 1 NGMRES corresponds to $m=0$, while depth 1 AA corresponds to $\hat m=1$.  For fair comparison we use depths of 1, 5, 20 for both AA and NGMRES.

Results are shown in Figure \ref{compare1}.  For both $Re$=400 and 1000,
%\todo{What is the norm you use? Maybe mention it in the paper} 
we note that Picard (i.e. with no acceleration) does not convergence and its nonlinear residual (in the $V'$ norm) remains O(1).  We observe for $Re$=400 that AA and NGMRES give very similar results for depths 5 and 20, but for depth 1 NGMRES is steadily reducing the residual and appears to eventually convergence while depth 1 AA is not making any progress towards convergence.  For $Re$=1000, only AA with depth 20 converges, while for NGMRES depth 20 convergens (slightly slower than for AA with depth 20) and depth 5 also converges (in about 120 iterations).  We conclude that for this problem, for larger depth AA and NGMRES have fairly similar convergence behavior, but for smaller depths NGMRES performs better.  We note these are initial tests and the authors plan more in depth comparisons of AA and NGMRES in future work.

\begin{figure}[h!]
\center
\includegraphics[width = .48\textwidth, height=.32\textwidth,viewport=0 0 660 420, clip]{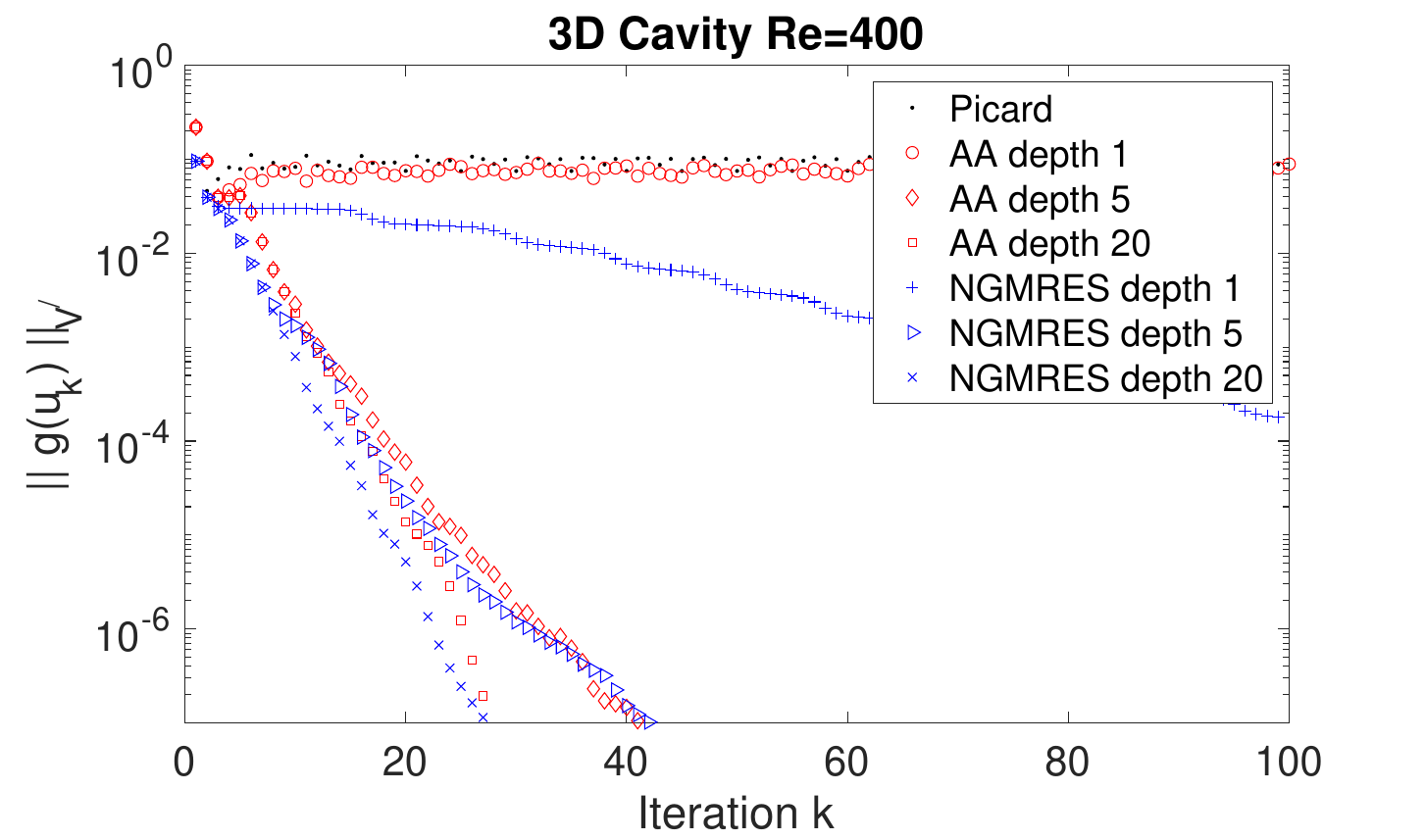}  
\includegraphics[width = .48\textwidth, height=.32\textwidth,viewport=0 0 660 420, clip]{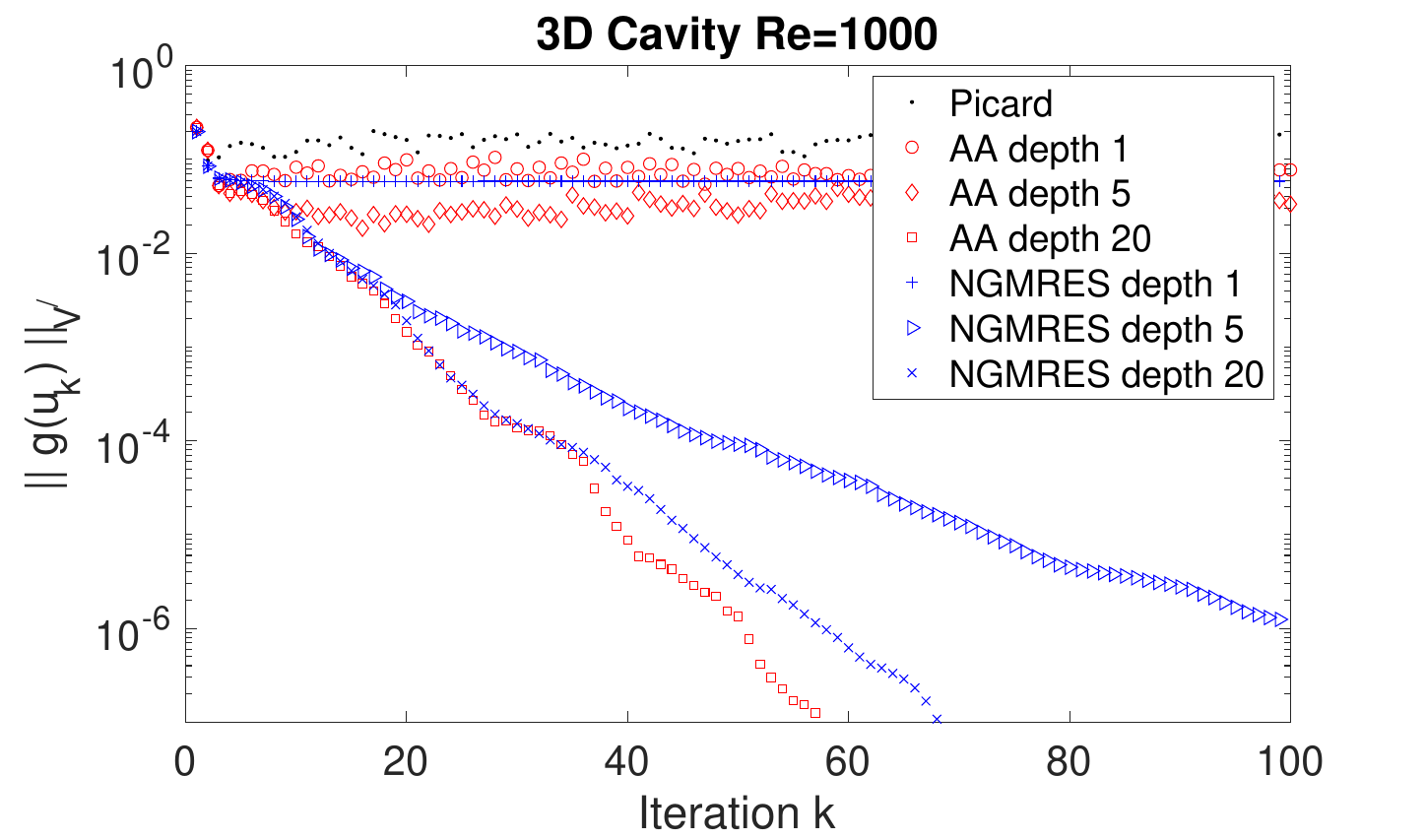} 
\caption{\label{compare1} The plots above show convergence of AA and NGMRES for the 3D driven cavity with varying $Re$ and extrapolation depth.}
\end{figure}

}

\subsection{Flow in a stenotic artery}

\begin{figure}[h!]
\center
\includegraphics[width = .9\textwidth, height=.27\textwidth,viewport=0 0 1200 300, clip]{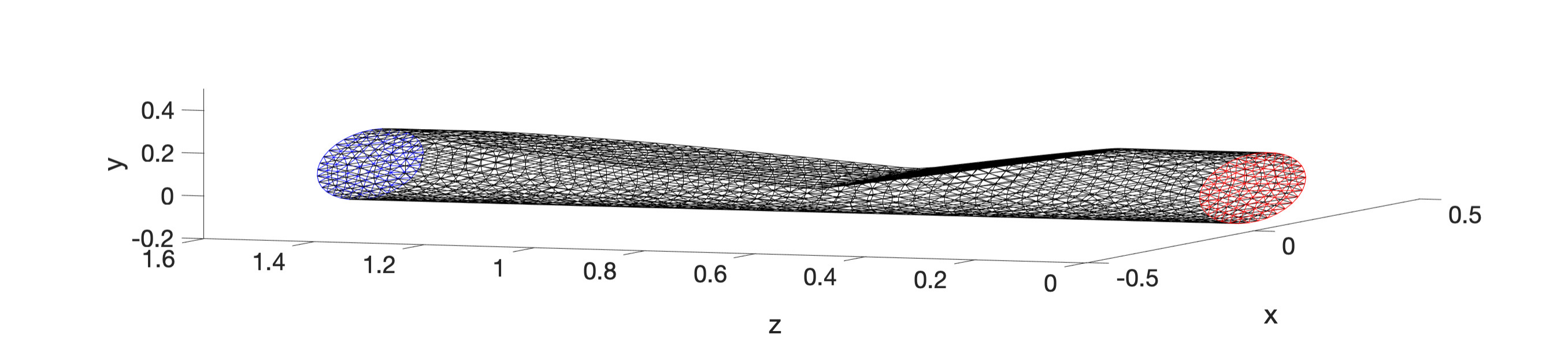}
\caption{\label{arterymesh} Shown above is the artery mesh (before the barycenter refinement is applied) restricted to the surface. }
\end{figure}

\begin{figure}[h!]
\center
\includegraphics[width = .9\textwidth, height=.27\textwidth,viewport=50 0 1200 300, clip]{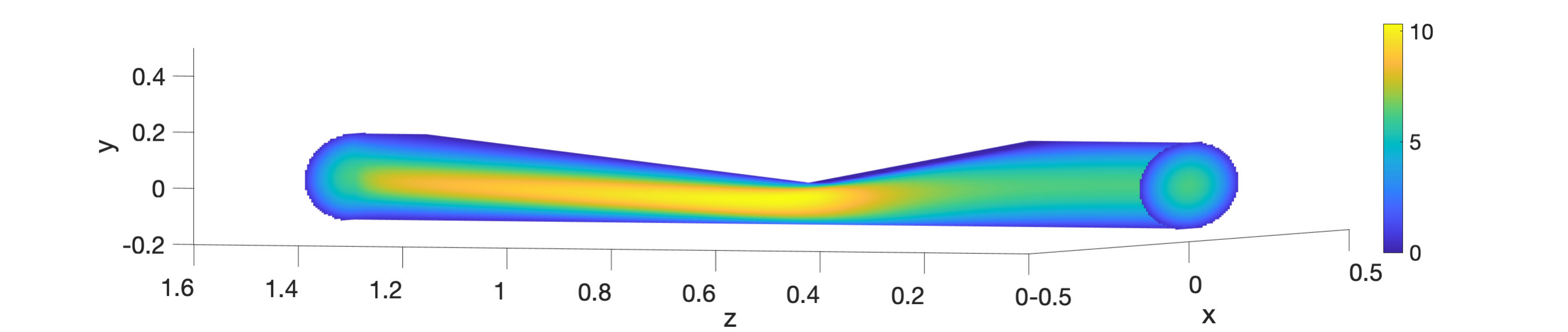}  
%% the original caption
% \caption{\label{artery100} Shown above are contour slices of the velocity magnitude for the $\nu=\frac{1}{500}$ solution found using Anderson accelerated Picard followed by Newton iterations until $H^1$ convergence of successive iterates to $10^{-10}$.}
% \end{figure}

\caption{\label{artery100} Shown above are contour slices of the velocity magnitude for the $\nu=\frac{1}{500}$ solution found using NGMRES accelerated Picard iterations until $H^1$ convergence of successive iterates to $10^{-10}$.}
\end{figure}

\begin{figure}[h!]
\center
\includegraphics[width = .48\textwidth, height=.32\textwidth,viewport=0 0 750 410, clip]{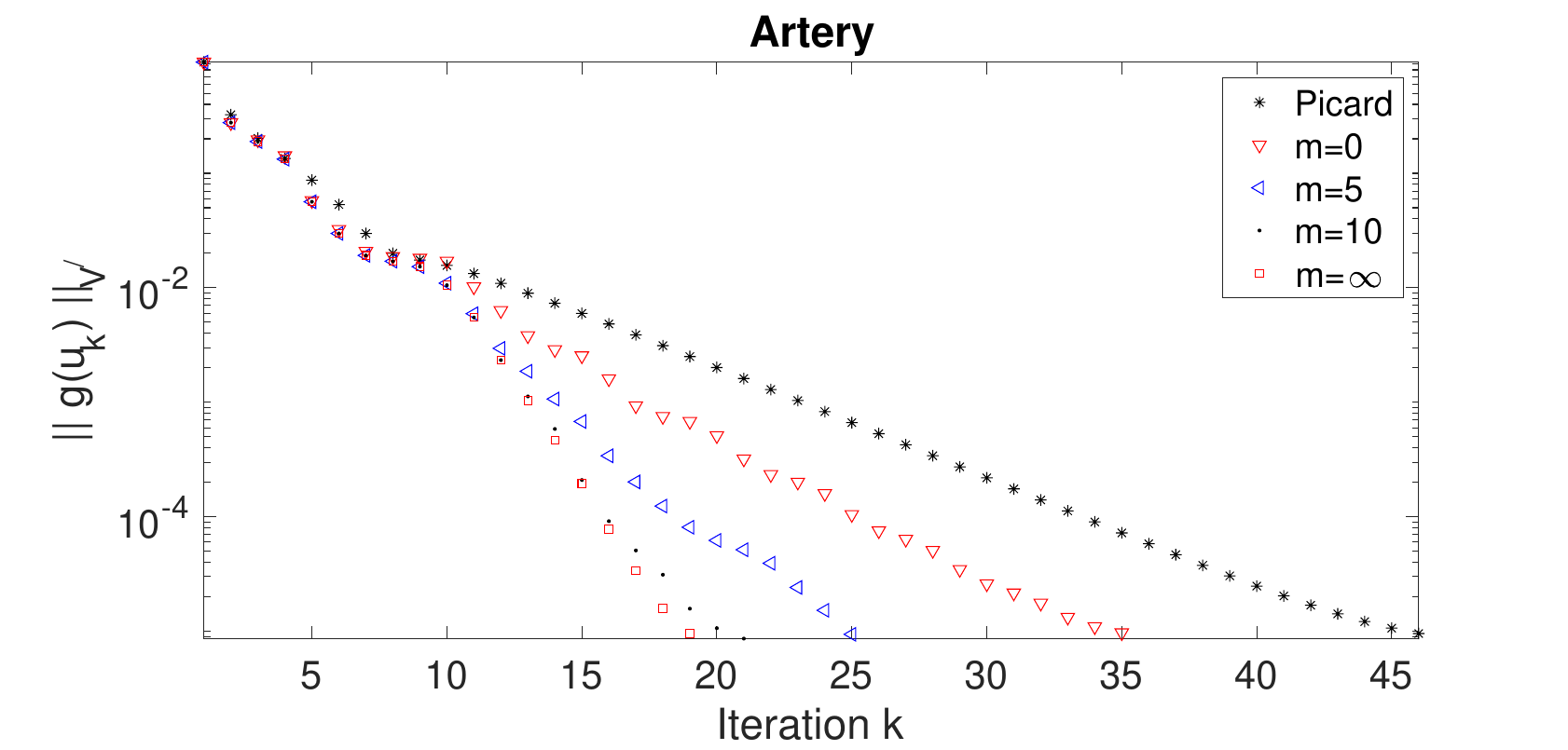}  
\includegraphics[width = .48\textwidth, height=.32\textwidth,viewport=0 0 750 410, clip]{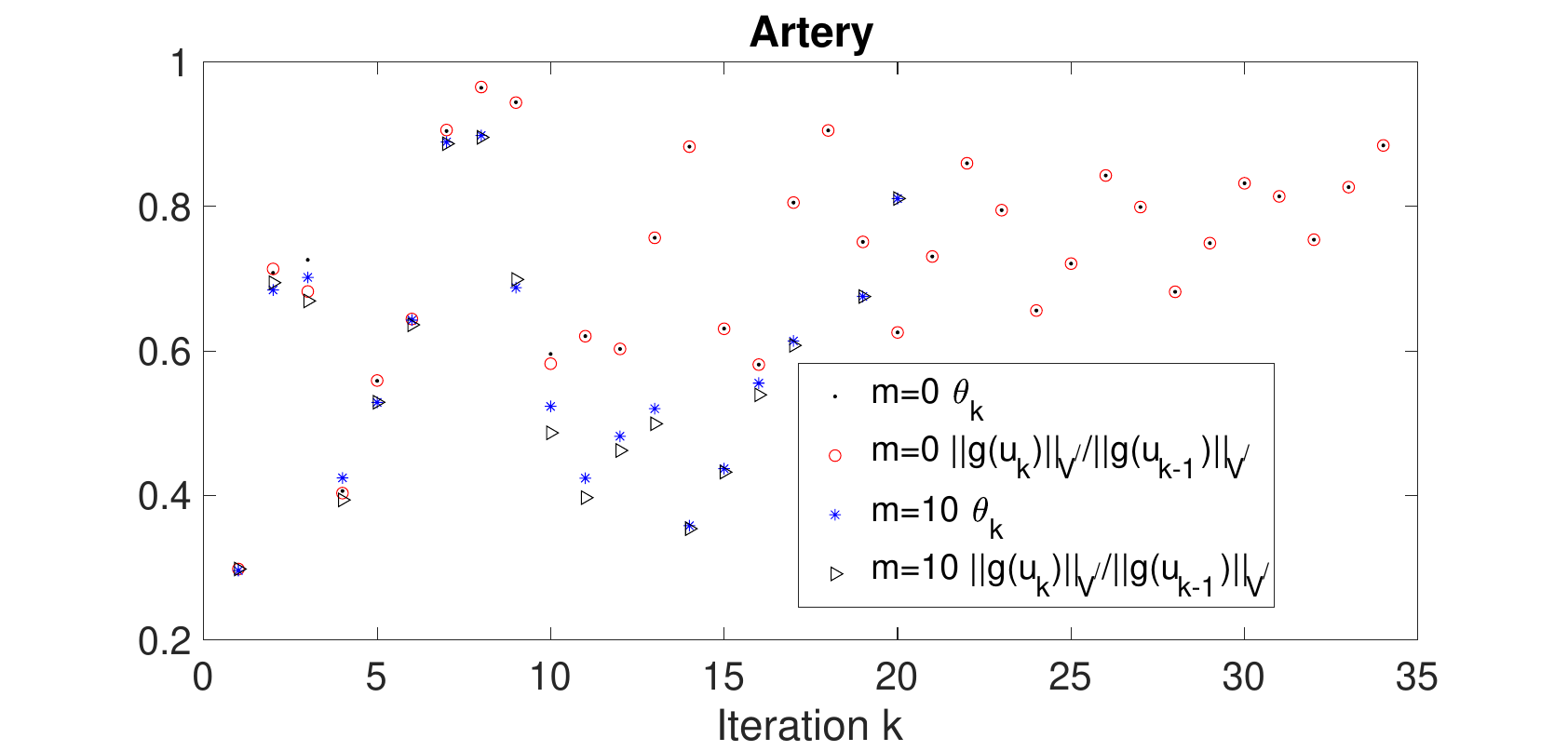}  \\
\includegraphics[width = .48\textwidth, height=.32\textwidth,viewport=0 0 750 410, clip]{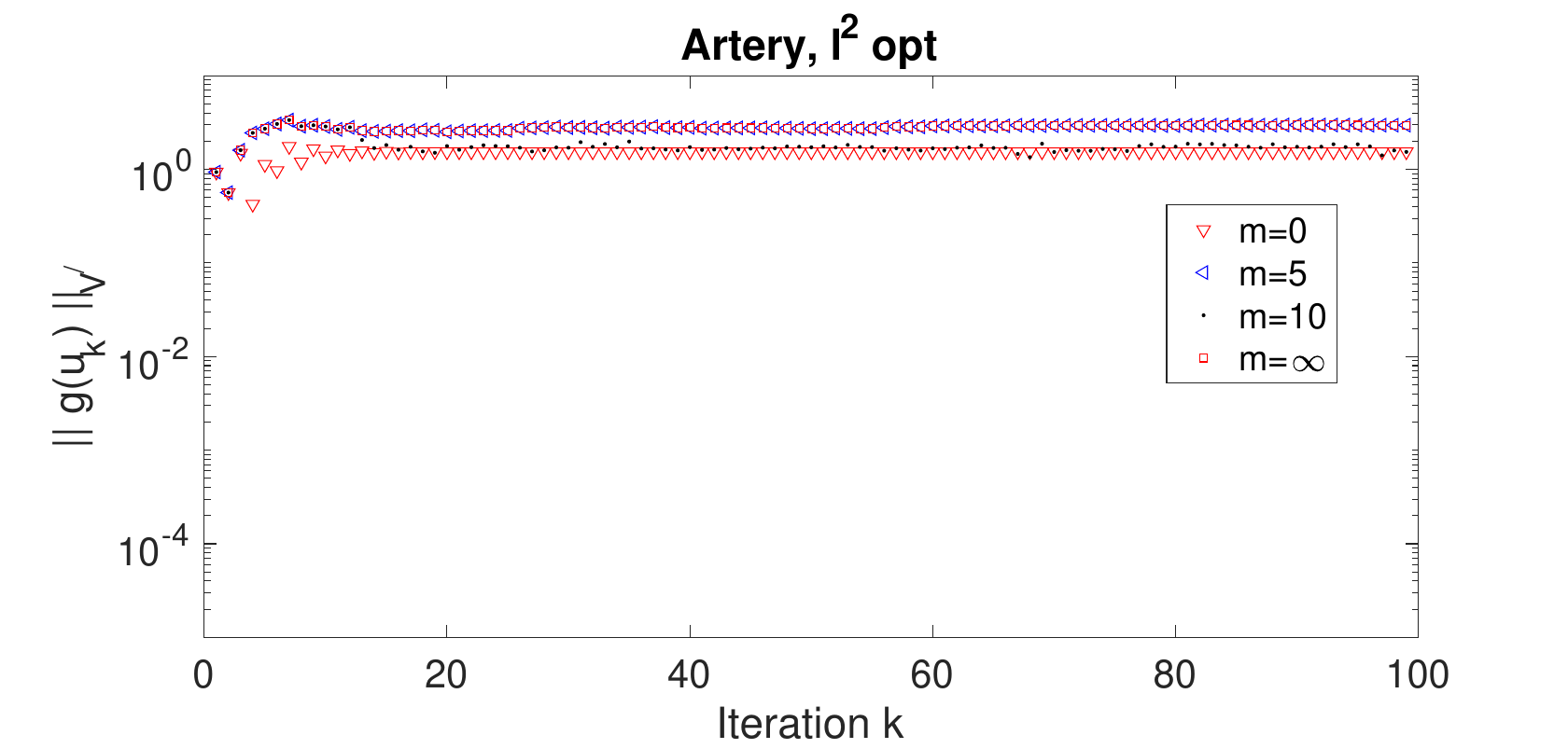}
\includegraphics[width = .48\textwidth, height=.32\textwidth,viewport=0 0 750 410, clip]{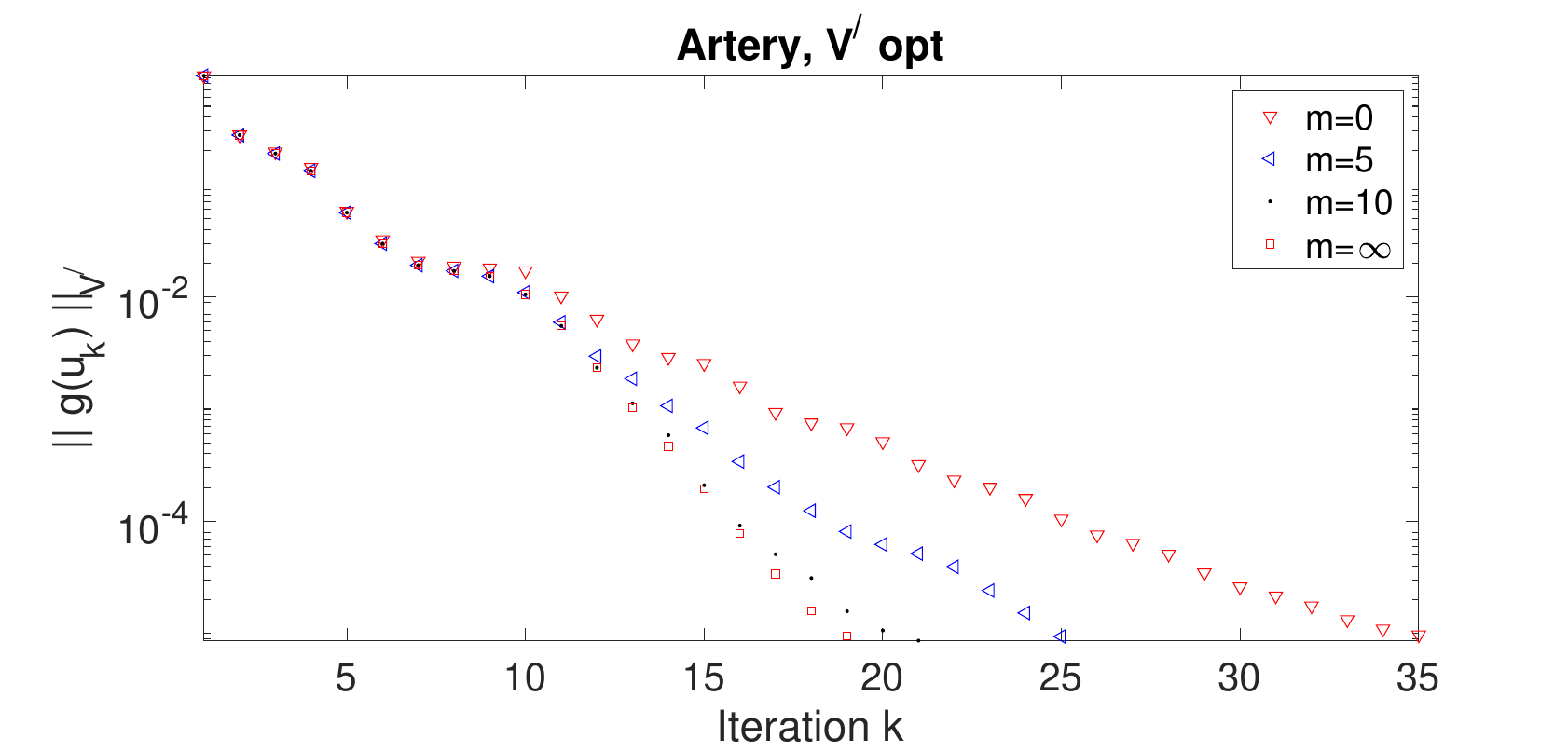}  
\caption{\label{artconv} The plots above show convergence behavior of NGMRES-Picard for 3D artery model, for varying $m$ (top left), $\theta_k$ and the ratio $\| g(u_k) \|_{V'} / \| g(u_{k-1}) \|_{V'}$ (top right), and at bottom is a comparison of convergence results using the $\ell^2$ optimization norm and $V'$ optimization norm.}
\end{figure}

Our third test problem is for a stenotic artery model from \cite{TKLV24}.  The vessel domain is created from a 3D cylinder with diameter $d=0.16cm$ and length $l=1.6cm$, with significant deformation through compressions on the vessel's top.  A plot of the domain is shown in Figure \ref{arterymesh} together with the (pre-barycenter refined) mesh (restricted to the surface) used for our computations with units in $cm$.  No-slip velocity boundary conditions are enforced on the artery walls, and Dirichlet parabolic inflow (at $z=0$) and outflow (at $z=1.6$) are enforced, using a maximum velocity at the vessel center of 6.22 cm/s (providing a 15mL/minute flow rate).  The viscosity is taken to be $\nu=0.002 g/(cm\cdot s)$.

The domain is discretized first from a regular {\color{black} tetrahedralization} (thanks to the authors of \cite{TKLV24} for providing the mesh) that is then barycenter refined, creating 52,912 total elements.  Once equipped with $(P_3,P_2^{disc})$ SV elements, there are 1.33M total dof.  A plot of the solution found on this mesh using NGMRES-Picard with $m=10$ and $V'$ optimization norm is shown in Figure \ref{artery100}.  Due to the mesh aspect ratio together with a linear solver tolerance of $10^{-10}$,  $\| \nabla u_k \|$ could only converge to about $10^{-5}$, which limited convergence of the nonlinear solver, since SV elements were used.  Thus we chose a convergence tolerance of $10^{-5}$ in the $V'$ norm for this problem.

Convergence plots are shown in Figure \ref{artconv} for Picard with no acceleration, and NGMRES-Picard using varying $m$ and the $V'$ optimization norm.  We observe that Picard converges in about 45 iterations and NGMRES-Picard converges in 34, 25, 21 and 18 iterations for $m=0,\ 5,\ 10,\ \infty$, respectively {\color{black}(noting that $m=\infty$ corresponds to choosing $m$ as large as possible, i.e. $m=k-1$ at iteration $k$).}  Hence the accelerated convergence is clear.  Also shown in the figure are plots of $\theta_k$ and $\frac{\| g(u_{k}) \|_{V'}}{\| g(u_{k-1})   \|_{V'}}$, and just as in all tests above, after the initial iterations these quantities agree quite well and showing the convergence estimate of Theorem \ref{thmm} is sharp.

Lastly, we compare convergence for NGMRES-Picard using the $\ell^2$ norm for the optimization to using $V'$.  Here we observe that NGMRES-Picard with $\ell^2$ fails to converge for each of $m=0,\ 5,\ 10, \ \infty$, while it converges well when $V'$ is used.  This is consistent with our discussion in Section 3 and tests above: in 3D, there is no reason to expect good results when using NGMRES-Picard for the NSE.

\section{Conclusions and Future Directions}
In this work we employ NGMRES as an acceleration technique for the Navier–Stokes Picard iteration, and provide a convergence analysis of the NGMRES–Picard iteration for general depth that shows NGMRES provides acceleration through the gain of the optimization problem. Our theoretical results suggest that using the dual norm of the divergence-free space is the appropriate choice in the optimization problem arising in the NGMRES algorithm; our results also indicate that different problems will likely have different best choices of optimization norm.  Numerical experiments for both 2D and 3D problems indicate that NGMRES significantly improves the performance of the Picard iteration, even in cases where the unaccelerated iteration diverges.

{\color{black}Extending the convergence analysis results herein to include stabilizations such as SUPG and also to general nonlinear problems with contractive and noncontractive associated solvers/iterations are important direction for better understanding NGMRES.  Further, while some initial comparisons of NGMRES to AA are given in this work, a more systematic comparison for this problem setting and beyond is a critical next step.  Additionally, in} \cite{he2026ngmresprecon}, the use of the underlying fixed point residual in the NGMRES optimization problem  and two new variants of AA based on reformulated least-squares problems using the Picard residual and nonlinear residual are proposed.  Hence another potential direction for future research is to explore convergence analyses of NGMRES when the $\ell^2$ norm is used in the optimization problem, given its simplicity, and the new variants of AA.

\bibliographystyle{siamplain}
\bibliography{graddiv}

 \end{document}